\documentclass[11pt,reqno]{amsart}
\usepackage{amssymb,tikz,rotating,tabularx,longtable,setspace,url,xcolor}
\usepackage[margin=1.375in]{geometry}
\usepackage{caption}
\usepackage{mathtools}
\DeclarePairedDelimiter\ceil{\lceil}{\rceil}
\DeclarePairedDelimiter\floor{\lfloor}{\rfloor}
\tikzset{vertex/.style={circle,draw,fill,scale=.35}}
\tikzset{op/.style={circle,fill,scale=.35}}
\tikzset{cl/.style={circle,fill,scale=.35}}
\tikzset{loopdown/.style={loop below,min distance=8mm,in=310,out=230,looseness=25}}
\tikzset{loopup/.style={loop above,min distance=8mm,in=130,out=50,looseness=25}}

\newcommand{\E}{\mathcal{E}}

\newcommand{\G}{\mathcal{G}}
\newcommand{\mybb}[1]{\mathbf{#1}}
\newcommand{\N}{\mathcal{N}}
\newcommand{\NN}{\mybb{N}}

\newcommand{\sage}{\textsc{SageMath}}
\newcommand{\V}{\mathcal{V}}
\newcommand{\W}{\mathcal{W}}
\newcommand{\wt}{\mathrm{wt}}

\usepackage{listings}
\lstdefinelanguage{Sage}[]{Python}
{morekeywords={False,True},sensitive=true}
\theoremstyle{plain}
\newtheorem{thm}{Theorem}[section]
\newtheorem{lemma}[thm]{Lemma}

\newtheorem{prop}[thm]{Proposition}

\theoremstyle{definition}
\newtheorem{dfn}[thm]{Definition}
\newtheorem{ex}[thm]{Example}
\newtheorem{remark}[thm]{Remark}
\numberwithin{equation}{section}
\numberwithin{figure}{section}
\numberwithin{table}{section}

\begin{document}
\title{Counting Anosov Graphs}
\author{Meera Mainkar}
\author{Matthew Plante}
\author{Ben Salisbury}

\address{Dept.\ of Mathematics\\ Central Michigan Univ.\\ Mt.\ Pleasant, MI 48859}

\email[Meera Mainkar]{maink1m@cmich.edu}
\email[Matthew Plante]{plant1mt@cmich.edu}
\email[Ben Salisbury]{salis1bt@cmich.edu}

\urladdr[Meera Mainkar]{http://people.cst.cmich.edu/maink1m}
\urladdr[Ben Salisbury]{http://people.cst.cmich.edu/salis1bt}

\thanks{M.M.\ was partially supported by CMU Early Career Grant \#C61940.}
\thanks{B.S.\ was partially supported by CMU Early Career Grant \#C62847.}


\begin{abstract}
In recent work by Dani and Mainkar, a family of finite simple graphs was used to construct nilmanifolds admitting Anosov diffeomorphisms.  Our main object of study is this particular set of graphs, which we call {\it Anosov graphs}.  Moreover, Dani and Mainkar give a lower bound on the number of Anosov graphs in terms of the number of vertices and number of edges.  In this work, we improve this lower bound in terms of vertices and edges, and we give lower and upper bounds solely in terms of the number of vertices. 
\end{abstract}

\maketitle

\section{Introduction}

We define an equivalence relation on the set of vertices of a simple graph as follows. We say that  two vertices are equivalent if they have the same open neighborhoods or the same closed neighborhoods.  We  call a connected simple graph  an {\em Anosov graph} if  each equivalence class contains at least two vertices and  the vertices in an equivalence class of size two are adjacent.  In \cite{DM}, this set of graphs was used to construct examples of nilmanifolds admitting Anosov diffeomorphisms  which play an important and beautiful role in hyperbolic dynamics.

To motivate the connection to Anosov diffeomorphisms, we first recall the construction of two-step nilpotent Lie algebra associated to a finite simple graph $G = (V, E)$,  where $V$ is the set of vertices and $E$ is the set of edges.  Denote by $\V$ the real vector space with basis $V$ and let $\W$ be the subspace of the second exterior power $\bigwedge^2\V$ spanned by $x \wedge y$, for $x, y \in V$ such that $xy \in E$.  Consider the (vector space) direct sum $\N = \V \oplus (\bigwedge^2 \V) / \W$.  One may define a Lie bracket structure on $\N$ by asserting
\begin{enumerate}
 \item $[v_1, v_2] = v_1 \wedge v_2 \bmod \W$, for $v_1, v_2 \in \V,$ and
 \item $[u, w] = 0$ for $u \in \N$ and $w \in (\bigwedge^2 \V)/ \W$.
\end{enumerate}

The Lie algebra $\N$ defined using the relations above is a two-step nilpotent Lie algebra; that is, $[\N, [\N, \N]] = \{0\}$. We call the simply-connected nilpotent Lie group $N$ corresponding to the Lie algebra $\N$ the {\em $($two-step$)$ nilpotent Lie group associated to the graph $G$}.  Let $\Gamma$ denote the lattice in $N$ corresponding to the (additive) subgroup of $\N$ generated by $V \cup \{\frac{1}{2} x \wedge y : xy \in E\}$. Then the nilmanifold $N /\Gamma$ is called a {\em $($two-step$)$ nilmanifold associated with graph $G$}.

In \cite{DM},  the authors proved that the graph is Anosov if and only if the two-step nilmanifold  associated with the graph admits an Anosov diffeomorphism.  In this article, we give a lower bound for the number of Anosov graphs in terms of number of vertices and number edges (see Theorem \ref{thm:matt}). We note that the dimension of the nilmanifold associated with a graph is the sum of number of vertices and number of edges.  It is known that two graphs are isomorphic if and only if the associated nilpotent Lie groups are isomorphic (see \cite{M}).  Hence Theorem \ref{thm:matt} gives a lower bound for number of nilmanifolds of a given dimension associated with graphs admitting Anosov diffeomorphisms.  This lower bound is an improvement of the lower bound given in \cite{DM}.

Whether or not a graph is Anosov involves looking at clusters of vertices and how they collectively interact with neighboring clusters of vertices.  The study  of  Anosov graphs may then be reduced to studying modified graphs with the clustering behavior of the vertices already simplified.  The resulting reduced graphs are called {\it quotient graphs}.

%

In Section \ref{boundsAnosov}, we give both lower and upper bounds on the number of Anosov graphs on $n$ vertices  (Theorem \ref{The big one}). In Section \ref{bound p(n)}, we prove that the number of Anosov graphs on $n$ vertices is bounded below by the number of partitions on $n$ if $n \geq 9$ (Theorem \ref{p(n)a(n)}). This shows that the number of Anosov graphs on $n$ vertices grows exponentially as $n$ increases.  For smaller values of $n$, we list all Anosov graphs on $n$ vertices in Appendix \ref{appA}.  In Appendix \ref{appB}, we give \sage\ code for determining whether or not a given graph is Anosov.

\section{Background and Notation}

\subsection{Partitions}
Denote the cardinality of a set $S$ by $|S|$ and let $[n]=\{1,2,\dots,n\}$ for all $n\in\NN$.
A \emph{partition} $\lambda$ of a natural number $n\in\NN$ is an ordered tuple $\lambda=(\lambda_1,\lambda_2,\dots,\lambda_k)$ where $\lambda_i\in \NN$, $\lambda_{i+1}\le\lambda_i$ for all $i\in[k]$ and $\lambda_1+\dots+\lambda_k=n$. If $\lambda=(\lambda_1,\dots,\lambda_k)$, define the length of $\lambda$ to be $k$, which we will denote by $\ell(\lambda)=k$. Write $\lambda\vdash n$ to indicate that $\lambda$ is a partition of $n$. If $\lambda=(\lambda_1,\dots,\lambda_k)$, then each $\lambda_i$ is called a \emph{part} of $\lambda$ and $\ell(\lambda)$ is the total number of parts in $\lambda$.

\subsection{Graphs}
We will appeal to \cite{Diestel} for common notions from the theory of graphs; any terms used, but not defined, here may be found in that text.

%

We will not consider graphs with multiple edges but adopt the convention that a graph is allowed loops and a {\it simple} graph is not allowed loops. A simple graph $G=(V,E)$ is \emph{complete} provided for each pair of distinct vertices $x,y\in V$ the edge $xy\in E$; these simple graphs are denoted $K^r$, where $r$ is the number of vertices in the graph. Similarly, a simple graph $G=(V,E)$ is \emph{edgeless} provided $E=\emptyset$; these graphs are denoted $\overline{K^r}$, where $r$ again denotes the number of vertices.  

Given an equivalence relation $R$ on the set of vertices $V$ of a graph $(V,E)$ the \emph{quotient graph} of $G$ modulo $R$ is the graph whose vertices are equivalence classes of the vertices induced by $R$ and there is an edge between the vertices if the vertices of the equivalence classes were adjacent in $G$. Denote the quotient graph $G$ modulo $R$ with 
\[
G/R=(V/R,\{P_iP_j:xy\in E \text{ for } x\in P_i \text{ and } y \in P_j\}),
\] 
where $\{P_1,\dots,P_k\}$ is the equivalence classes in $V/R$. If multiple edges exist between vertices, then merge them into a single edge.

\subsection{Neighborhoods}\label{nhd}

Let $G=(V,E)$ be a graph. For $x \in V$, define the \emph{open neighborhood} and \emph{closed neighborhood} of $x$, respectively, to be
\[
N(x) = \{ y \in V : xy \in E \} \qquad\text{ and } \qquad
\overline{N}(x) = N(x) \cup \{x \}.
\]

\begin{dfn}
For vertices $x,y\in V$, we say $x \sim y$ if $N(x)=N(y)$ or $\overline{N}(x)=\overline{N}(y)$.  
\end{dfn}

\begin{remark}
If $N(x) = N(y)$, it is sometimes said that $x$ and $y$ are twin vertices.  Along these lines, the condition $\overline{N}(x)=\overline{N}(y)$ may be interpreted as $x$ and $y$ being ``conjoined twins.''
\end{remark}

\begin{lemma}
\label{eqrel}
The relation $\sim$ is an equivalence relation on $V$.
\end{lemma}
\begin{proof}
One can observe $\sim$ is clearly reflexive. Also, $\sim$ is symmetric, as for any $x,y\in V$ if $x\sim y$ then either $N(x)=N(y)$ or $\overline{N}(x)=\overline{N}(y)$.  In either case $y\sim x$. To see $\sim$ is transitive, assume $x\sim y$ and $y\sim z$ for distinct $x,y,z\in V$. If $N(x)=N(y)=N(z)$ or $\overline{N}(x)=\overline{N}(y)=\overline{N}(z)$ then we are done. On the contrary, assume $N(x)=N(y)$ and $\overline{N}(y)=\overline{N}(z)$.  Then $z\in N(x)$, or rather $x\in N(z)$ and thus $x\in N(y)$ a contradiction since we assumed $N(x)=N(y)$.
\end{proof}

For any simple graph $G$ we can partition the set of vertices $V$ into equivalence classes, $P_1,P_2,\dots,P_k$, with respect to $\sim$.
Let $\lambda_i=|P_i|$ for each $i\in[k]$ and assume that $\lambda_1 \geq \cdots \geq \lambda_k$. Then each vertex of $V$ is in exactly one $P_i$ and
$\lambda=(\lambda_1,\lambda_2,\dots,\lambda_k)\vdash |V|$.
In this case, we say $G$ is of \emph{type} $\lambda=(\lambda_1,\dots,\lambda_k)$. We will say a class $P_i$ is \emph{edgeless} if the subgraph $G[P_i]$ is edgeless and we will say $P_i$ is \emph{complete} if the subgraph $G[P_i]$ is complete.

\begin{ex}
Suppose $G=(V,E)$ is the graph
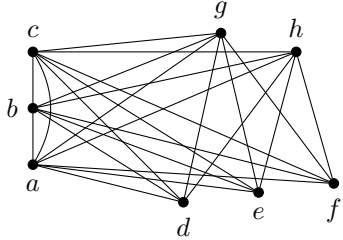
\[
\begin{tikzpicture}[font=\small,baseline=-4]
\node[vertex,label={below:$a$}] (1) at (0,0) {};
\node[vertex,label={left:$b$}] (2) at (0,.75) {};
\node[vertex,label={above:$c$}] (3) at (0,1.5) {};

\node[vertex,label={below:$d$}] (4) at (2,-.5) {};
\node[vertex,label={below:$e$}] (5) at (3,-.37) {};
\node[vertex,label={below:$f$}] (6) at (4,-.25) {};

\node[vertex,label={above:$g$}] (7) at (2.5,1.75) {};
\node[vertex,label={above:$h$}] (8) at (3.5,1.5) {};
\path[-]
 (1) edge (2)
 (1) edge [bend right] node[left] {} (3)
 (1) edge (4)
 (1) edge (5)
 (1) edge (6)
 (1) edge (7)
 (1) edge (8)

 (2) edge (3)
 (2) edge (4)
 (2) edge (5)
 (2) edge (6)
 (2) edge (7)
 (2) edge (8)

 (3) edge (4)
 (3) edge (5)
 (3) edge (6)
 (3) edge (7)
 (3) edge (8)

 (4) edge (7)
 (4) edge (8)

 (5) edge (7)
 (5) edge (8)

 (6) edge (7)
 (6) edge (8);
\end{tikzpicture}
\]

Here the relation $\sim$ partitions $V$ into $P_1 = \{ a,b,c\}$, $P_2 = \{d,e,f\}$, $P_3 = \{g,h\}$, so $G$ is of type $\lambda=(3,3,2)$.  Moreover $G[P_1]$ is complete, $G[P_2]$ is edgeless, and $G[P_3]$ is edgeless.
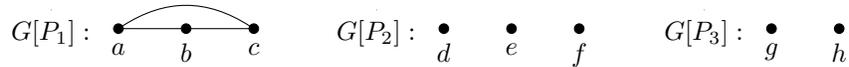
\[
\begin{tikzpicture}[font=\small,baseline=-4,xscale=.9]
\node[vertex,label={below:$a$}] (1) at (1,0) {};
\node[vertex,label={below:$b$}] (2) at (2,0) {};
\node[vertex,label={below:$c$}] (3) at (3,0) {};
\node[circle,fill,scale=0,label={below:$G[P_1]:$}] (10) at (0,.25) {};
\path[-]
 (1) edge (2)
 (1) edge [bend left] node[left] {} (3)
 (2) edge (3);
\end{tikzpicture}
\hspace{2em}
\begin{tikzpicture}[font=\small,baseline=-4,xscale=.9]
\node[vertex,label={below:$d$}] (1) at (1,0) {};
\node[vertex,label={below:$e$}] (2) at (2,0) {};
\node[vertex,label={below:$f$}] (3) at (3,0) {};
\node[circle,fill,scale=0,label={below:$G[P_2]:$}] (10) at (0,.25) {};
\end{tikzpicture}
\hspace{2em}
\begin{tikzpicture}[font=\small,baseline=-4,xscale=.9]
\node[vertex,label={below:$g$}] (1) at (1,0) {};
\node[vertex,label={below:$h$}] (2) at (2,0) {};
\node[circle,fill,scale=0,label={below:$G[P_3]:$}] (10) at (0,.25) {};
\end{tikzpicture}
\]
The quotient graph $G/{\sim}$ is the graph,
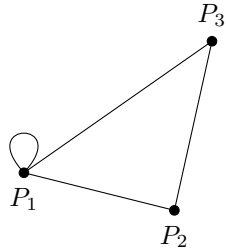
\[
\begin{tikzpicture}[font=\small,baseline=-4]
\node[vertex,label={below:$P_1$}] (1) at (0,0) {};

\node[vertex,label={below:$P_2$}] (4) at (2,-.5) {};

\node[vertex,label={above:$P_3$}] (7) at (2.5,1.75) {};
\path[-]
 (1) edge (4)
 (1) edge (7)
 (1) edge [loopup] (1)
 (4) edge (7);
\end{tikzpicture}
\]
\end{ex}
\begin{lemma}
Each equivalence class $P_i$ is either edgeless or complete.
\end{lemma}

\begin{proof}
If $P_i$ is edgeless, then we are done. On the other hand, suppose $P_i$ is not edgeless; that is, there is an edge $xy\in E(P_i)$.
To show $P_i$ is complete we will have to show for every two vertices $u,v\in V(P_i)$, $uv\in E(P_i)$.
Let $u$ and $v$ be $V(P_i)$. If $u=x$ and $v=y$ then we are done, so assume otherwise.
Without loss of generality, say $u\neq x$.
Since $u$ and $y$ are both in $P_i$, we know $u\sim y$ and $N(y)\subseteq \overline N(u)$.
However $x\in N(y)$ since $xy\in E(P_i)$, thus $x\in \overline N(u)$ and $xu \in E(P_i)$ since $x\neq u$.
Since $v \sim x$ and $xu \in E(P_i)$, we have $uv\in E(P_i)$.
Since $u$ and $v$ are chosen arbitrarily, it follows that $P_i$ is complete.
\end{proof}

\section{Anosov graphs}

\begin{dfn}
Let $G$ be a connected simple graph of type $\lambda=(\lambda_1,\dots,\lambda_k)$. We say $G$ is \emph{Anosov} if
\begin{enumerate}
\item $\lambda_k \ge 2,$ and
\item $\lambda_i = 2$ implies $P_i$ is edgeless for all $1\leq i\leq k$, where $P_i$ denotes the corresponding equivalence class (see Section \ref{nhd}).
\end{enumerate}
\end{dfn}

\begin{ex}
The graph
\[
\begin{tikzpicture}[font=\small,baseline=-4]
\node[vertex,label={below:$v_1$}] (1) at (0,0) {};
\node[vertex,label={below:$v_2$}] (2) at (2,0) {};
\node[circle,fill,scale=0,label={below:$G:$}] (6) at (-1,.5) {};
\node[vertex,label={below:$v_3$}] (3) at (4,0) {};
\node[vertex,label={above:$v_4$}] (4) at (1,1) {};
\node[vertex,label={above:$v_5$}] (5) at (3,1) {};
\path[-]
 (1) edge (4)
 (1) edge (5)
 (2) edge (4)
 (2) edge (5)
 (3) edge (4)
 (3) edge (5);
\end{tikzpicture}
\]
is Anosov, but the two graphs
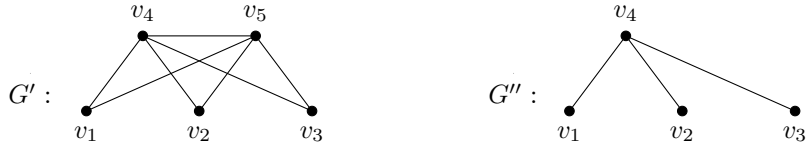
\[
\begin{tikzpicture}[font=\small,baseline=-4,xscale=.75]
\node[vertex,label={below:$v_1$}] (1) at (0,0) {};
\node[vertex,label={below:$v_2$}] (2) at (2,0) {};
\node[circle,fill,scale=0,label={below:$G':$}] (6) at (-1,.5) {};
\node[vertex,label={below:$v_3$}] (3) at (4,0) {};
\node[vertex,label={above:$v_4$}] (4) at (1,1) {};
\node[vertex,label={above:$v_5$}] (5) at (3,1) {};
\path[-]
 (1) edge (4)
 (1) edge (5)
 (2) edge (4)
 (2) edge (5)
 (3) edge (4)
 (4) edge (5)
 (3) edge (5);
\end{tikzpicture}
\hspace{5em}
\begin{tikzpicture}[font=\small,baseline=-4,xscale=.75]
\node[vertex,label={below:$v_1$}] (1) at (0,0) {};
\node[vertex,label={below:$v_2$}] (2) at (2,0) {};
\node[circle,fill,scale=0,label={below:$G'':$}] (6) at (-1,.5) {};
\node[vertex,label={below:$v_3$}] (3) at (4,0) {};
\node[vertex,label={above:$v_4$}] (4) at (1,1) {};
\path[-]
 (1) edge (4)
 (2) edge (4)
 (3) edge (4);
\end{tikzpicture}
\]
are not Anosov.
\end{ex}

Let $\mathcal{A}$ be the family of Anosov graphs, $\mathcal{A}_n$ be the set of all graphs in $\mathcal{A}$ on $n$ vertices, and let $a(n)$ be the number of Anosov graphs on $n$ vertices. We will now define families $\mathcal{U}$ and $\mathcal{L}$ which will be useful in finding a lower and upper bound on the number of Anosov graphs in terms of vertices. Define $\mathcal{U}$ to be the set of all  (not necessarily connected) graphs of type $\lambda=(\lambda_1,\dots,\lambda_k)$ such that $\lambda_k\geq 2$, $\mathcal{U}_n$ to be the set of all graphs in $\mathcal{U}$ on $n$ vertices, and $U(n)=|\mathcal{U}_n|$. Similarly, define $\mathcal{L}$ to be the set of all connected graphs of type $\lambda=(\lambda_1,\dots,\lambda_k)$ such that $\lambda_k\geq 3$, $\mathcal{L}_n$ to be the set of all graphs in $\mathcal{L}$ on $n$ vertices, and $L(n)=|\mathcal{L}_n|$. By definition of Anosov graphs, we have the following.

\begin{prop}
\label{UpandLow}

$L(n)\leq a(n)\leq U(n)$.
\end{prop}



\section{Quotient graphs and lower bounds}\label{brick}


In this section we shift our attention to the quotient graphs of simple graphs modulo $\sim$ (see subsection \ref{nhd}). Quotient graphs are particularly useful when observing Anosov graphs since each Anosov graph partitions its vertices into classes of size at least 2. We associate with each of these quotient graphs $G/{\sim}$ a weight function, $\wt\colon V/{\sim}\longrightarrow\NN$, such that $\wt(\mathbf{v})$ is the number of vertices from $G$ in the same equivalence class of $v$ induced by $\sim$. In drawing the quotient graph induced by $\sim$ we write each weight next to its corresponding vertex. Another important question we will address first is when is a graph $\G=([k],\E,\wt)$ where $\wt$ is an arbitrary weight function $\wt\colon [k]\longrightarrow\E$ a quotient graph, $G/{\sim}$, for some simple graph $G$. By answering this question we also develop a method for constructing our first lower bound on Anosov graphs in terms of both vertices and edges.

\begin{dfn}\label{mudtog}
Given a graph $\G=([k],\E,\wt)$, we say the \emph{deconstruction} of $\G$ induced by $\sim$
is a graph $G=(V,E)$ defined as follows.
For each $i\in [k]$, define
\begin{equation}\label{subset}
D_i=\{v_{i,1},\dots, v_{i,\wt(i)}\}
\qquad \text{ and } \qquad
V=\bigsqcup_{i\in[k]}D_i.
\end{equation}
For  distinct $x,y\in V$ with $x\in D_i$ and $y\in D_j$ for some $i,j\in[k]$, set
\begin{enumerate}
\item $xy\in E$ if $ij\in\E$,
\item $xy\not\in E$ otherwise.
\end{enumerate}
\end{dfn}


\begin{lemma}\label{advsim}
Let $\G=([k],\E,\wt)$ and let $G=(V,E)$ be the deconstruction of $\mathcal{G}$, with $D_1,D_
2,\dots, D_k$ subsets of $V$ as defined in Equation \eqref{subset}. If $x,y\in D_i$ for any $i\in[k]$, then $x\sim y$.
\end{lemma}
\begin{proof}
Let $x,y\in D_i$ be given. We will show that $N(x) = N(y)$ or $\overline{N}(x) = \overline{N}(y)$. Let $z \in N(x)$.
We will break this proof into two cases.

\underline{Case 1.} Assume $z\in D_i.$ Since $z\in N(x)$, it must be that $\overline{N}(x)=\overline{N}(z)$ by Definition \ref{mudtog}. However, $y\in D_i$, so this would imply  $\overline{N}(y)=\overline{N}(z)$. Combining these we get  $\overline{N}(x)=\overline{N}(y)$.

\underline{Case 2.} Assume $z\not\in D_i.$  Since $z\not\in D_i$ and $z\in N(x)$, it must be that $z\in D_j$ for some $j\neq i$ with $ij\in\E$. However, $y\in D_i$, so this would imply $z\in N(y)$. Therefore $N(x) \subseteq  N(y)$.
Similarly we can show that $N(y) \subseteq  N(x)$
 and hence $x\sim y$ for all $x,y\in D_i$ for all $i\in[k]$.
\end{proof}

\begin{remark}
As a function from graphs with weight functions to simple graphs, deconstruction is not injective.
\end{remark}


\begin{thm}
\label{Anosov Criteria}
A connected graph $\G=([k],\E,\wt)$ has a deconstruction to an Anosov graph provided
\begin{enumerate}
\item $\wt(i)\ge 2$ for all $i\in [k]$;
\item for all $i\in [k]$, if $\wt(i)=2$, then $ii\not\in \E$.
\end{enumerate}
\end{thm}
\begin{proof} Assume the above conditions for $\G=([k],\E,\wt)$. Let $G=(V,E)$, be the deconstruction of $\G$ with $D_1,D_2,\dots, D_k$ subsets of $V$ as defined in Equation \eqref{subset}. Since $\G$ is connected,  $G$ is also connected. If $\sim$ partitions the vertices of $V$ into classes $P_1,\dots,P_m$, then by Lemma \ref{advsim} for each $i\in [k]$, there exists $j\in[m]$ such that $D_i\subseteq P_j$. Since each $P_j$ is nonempty, it must contain a $D_i$, so we get a well-defined surjective function $\{D_1,\dots,D_k\}\longrightarrow \{P_1,\dots,P_m\}$. Thus $m\leq k$. Since for each $i\in [k]$ there exists $j\in[m]$ such that $D_i\subseteq P_j$, we obtain
\[
\min(|P_1|,\dots,|P_m|)\geq\min(|D_1|,\dots,|D_k|)=\min(\wt(1),\dots,\wt(k)).
\]
 Since we assumed $\wt(i) \ge 2$ for all $i\in[k]$ we have $|P_j|\ge 2$ for all $j\in[m]$. If $|P_j|=2$ then $|D_i|=2$ with $D_i=P_j$ for some $i\in[k]$ and $P_j$ is edgeless by condition (2), since $G[D_i]$ is edgeless so $G$ is Anosov.
\end{proof}

\begin{remark} Conditions (1) and (2) in Theorem \ref{Anosov Criteria} are not necessary conditions for a deconstruction of a graph with a weight function to be Anosov.
\end{remark}

\begin{lemma}
\label{The Brick Graph Criteria}
A graph $\G=([k],\E,\wt)$ is the quotient graph $G/{\sim}$ where $G$ of type $(\wt(1),\dots,\wt(k))$ provided, if $i,j\in[k]$ are distinct, then $ii\in\E$ and $jj\in \E$ implies $\overline{N}_{\mathcal{G}}(i)\neq\overline{N}_{\mathcal{G}}(j)$; similarly, $ii\not\in\E$ and $jj\not\in\E$ implies $N_{\mathcal{G}}(i)\neq N_{\mathcal{G}}(j)$.  In this case, $G/{\sim}=\mathcal{G}$.
\end{lemma}

\begin{proof}
Assume the above condition for $\G=([k],\E,\wt)$. Let $G=(V,E)$ be the deconstruction of $\mathcal{G}$, with $D_1,D_2,\dots, D_k$ subsets of $V$ as defined in Equation \eqref{subset}.
Assume $\{D_1,\dots, D_k\}$ are the equivalence classes coming from $\sim$. Notice for $i,j\in[k]$, by Definition \ref{mudtog}, we have $ij\in \mathcal{E}$, $x\in P_i$, $y\in P_j$ and $x\in N(y)$, $\wt(i)=\lambda_i$, and $ii\not\in \E$ when $G[D_i]$ is edgeless and $ii\in \E$ when $G[D_j]$ is complete. We claim that $\{D_1,\dots,D_k\}$ are equivalence classes coming from $\sim$. If we can prove this claim, then $\mathcal{G}$ is the quotient graph $G/{\sim}$.
By Lemma \ref{advsim}, we have that $x\sim y$ for all $x,y\in D_i$ for each $i\in[k]$. It suffices to show $x\not\sim y$ provided $i\neq j$ with $x\in D_i$ and $y\in D_j$ for all $i,j\in[k]$.
Suppose, by way of contradiction, that $i\neq j$ but there exists $x\in D_i$ and $y\in D_j$ such that $x\sim y$. We will break this proof into two cases.

\underline{Case 1.} Assume $xy\in E$. Since $xy\in E$ by our deconstruction $ij\in\E$. Let $z\in D_i$, then by transitivity of $\sim$ we have $z\sim y$. Combining $xy\in E$ and $z\sim y$ we have $xz\in E$. Since $G[D_i]$ is not edgeless $G[D_i]$ is complete. By a similar argument we have $G[D_j]$ is complete. Since $G[D_i]$ is complete and $G[D_j]$ is complete we have $ii\in \E$ and $jj\in\E$. However $x\sim y$ so by our construction $i\sim j$ in $\mathcal{G}$ and $\overline N_{\mathcal{G}}(i)=\overline N_{\mathcal{G}}(j)$. This contradicts our hypothesis.

\underline{Case 2.} Assume $xy\not\in E$. Since $xy\not \in E$ by our deconstruction $ij\not\in \E$. Let $z\in D_i$, then by transitivity of $\sim$ we have $z\sim y$. Combining $xy\not\in E$ and $z\sim y$ we have $xz\not\in E$. Since $xz\in E$ we have $G[D_i]$ is edgeless. By a similar argument we have $G[D_j]$ is edgeless. Since $G[D_i]$ is edgeless and $G[D_j]$ is edgeless we have $ii\not\in\E$ and $jj\not\in\E$. However $x\sim y$ so by our construction $i\sim j$ in $\mathcal{G}$  and $N_{\mathcal{G}}(i)=N_{\mathcal{G}}(j)$. This contradicts our hypothesis.

In either case, we arrive at a contradiction. Therefore, $\mathcal{G}$ is the quotient graph $G/{\sim}$.
\end{proof}

\begin{thm}
\label{Brick}
A connected graph $\G=([k],\E,\wt)$ is a quotient graph $G/{\sim}$ where $G$ is Anosov of type $(\wt(1),\dots,\wt(k))$ provided,
\begin{enumerate}
\item $\wt(k)\ge 2$;
\item for all $i\in[k]$, if $\wt(i)=2$, then $ii\not\in\E$;
\item if $i,j\in[k]$ are distinct, then $ii\in\E$ and $jj\in \E$ implies $\overline{N}_{\mathcal{G}}(i)\neq\overline{N}_{\mathcal{G}}(j)$; similarly, $ii\not\in\E$ and $jj\not\in\E$ implies $N_{\mathcal{G}}(i)\neq N_{\mathcal{G}}(j)$.
\end{enumerate}
\end{thm}

\begin{proof} By Theorem \ref{Anosov Criteria}, $G$ is Anosov of type $(\wt(1),\dots,\wt(k))$, and by Lemma \ref{The Brick Graph Criteria}, $G/{\sim}=\mathcal{G}$.
\end{proof}

Now we will improve the lower bound given in \cite{DM}, which is
\begin{equation}\label{firstbound}
\nu(n+m)>\frac{1}{12}(n+m-3\sqrt{2(n+m)}-17),
\end{equation}
where $\nu(n+m)$ is the number of Anosov graphs with $n$ vertices and $m$ edges.

\begin{thm}\label{thm:matt}
The number of Anosov graphs is bounded below in terms of the number of vertices, $n$, and number of edges, $m$, that is,
\begin{equation}\label{lowerbound}
\nu(n+m)> \frac{n+m}{3}-\sqrt{2(n+m)}-\frac{9}{2}.
\end{equation}
\end{thm}

\begin{proof}
Let $\G=([k],\E,\wt)$, such that $\wt(0)=q\geq 2$, $\wt(i)=2$ for all $i\in[k]$, $ii\not\in\E$ for all $i\in\{0\}\cup[k]$, and there is only one edge connecting the vertex $0$ to the graph.  By Theorem \ref{Anosov Criteria}, if $G=(V,E)$ is the deconstruction of $\G$ then $G$ is Anosov. Let $p$ be the number of edges in $\mathcal{G}[\{1,\dots,k\}]$, but there is only one edge connecting $0$ to the rest of $\mathcal{G}$ so $p=|\E|-1$. For $i, j \in [k]$, consider the induced subgraph $\mathcal{G}[i,j]$ with $ij\in \mathcal{E}$:
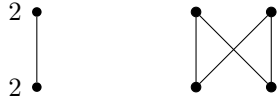
\[
\begin{tikzpicture}[font=\small,baseline=-4]
\node[op,label={left:$2$}] (1) at (0,.5) {};
\node[op,label={left:$2$}] (2) at (0,-.5) {};
\path[-]
 (1) edge (2);
\end{tikzpicture}
\hspace{5em}
\begin{tikzpicture}[font=\small,baseline=-4]
\node[vertex,label={below:$$}] (1) at (0,.5) {};
\node[vertex,label={below:$$}] (2) at (0,-.5) {};
\node[vertex,label={below:$$}] (3) at (1,.5) {};
\node[vertex,label={below:$$}] (4) at (1,-.5) {};
\path[-]
 (1) edge (2)
 (1) edge (4)
 (3) edge (2)
 (3) edge (4) ;
\end{tikzpicture}
\]
From this we notice between two connected vertices in $\mathcal{G}[\{1,\dots,k\}]$ are 4 edges in $\mathcal{D}(\mathcal{G})$. Similarly consider the subgraph  $\mathcal{G}[0,i]$ where $0i\in\E$:
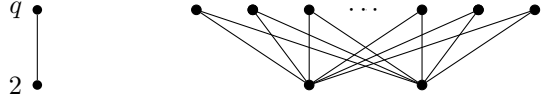
\[
\begin{tikzpicture}[font=\small,baseline=-4]
\node[op,label={left:$q$}] (1) at (0,1) {};
\node[op,label={left:$2$}] (2) at (0,0) {};
\path[-]
 (1) edge (2);
\end{tikzpicture}
\hspace{5em}
\begin{tikzpicture}[font=\small,baseline=-4,xscale=1.5]
\node[vertex,label={below:$$}] (1) at (.5,0) {};
\node[vertex,label={below:$$}] (2) at (-.5,0) {};
\node[vertex,label={below:$$}] (3) at (1.5,1) {};
\node[vertex,label={below:$$}] (4) at (-1.5,1) {};

\node[vertex,label={center:$$}] (5) at (1,1) {};
\node[circle,fill,scale=0,label={center:$\cdots$}] (6) at (0,1) {};
\node[vertex,label={center:$$}] (11) at (-1,1) {};

\node[vertex,label={below:$$}] (7) at (-1,1) {};
\node[vertex,label={below:$$}] (8) at (-.5,1) {};
\node[vertex,label={below:$$}] (9) at (.5,1) {};
\node[vertex,label={below:$$}] (10) at (1,1) {};

\path[-]
 (1) edge (3)
 (1) edge (4)
 (2) edge (3)
 (2) edge (4)
 (1) edge (7)
 (1) edge (8)
 (1) edge (9)
 (1) edge (10)
 (2) edge (7)
 (2) edge (8)
 (2) edge (9)
 (2) edge (10) ;
\end{tikzpicture}
\]
Here we notice the number of edges between the vertex $0$ and the vertex $0$ is connected with is $2q$ in $G$. Hence the number of edges, $m$, in $G$ is $4p+2q$ and the number of vertices, $n$, is $2k+q$, or rather, $n+m=2k+3q+4p$. Let $w=n+m$. We will assume $w>9$. Now we will construct  Anosov graphs for given $w=n+m$.  We choose $q=2$ if $w$, is even and $q=3$ if $w$ is odd. In the case where $w$ is even we have $2k+3q+4p=2k+6+4p$, so $w=2k+6+4p$, or rather,
\[
\frac{w-6}{2}=k+2p.
\]
Similarly if $w$ is odd,
\[
\frac{w-9}{2}=k+2p.
\]
For simplicity,
\[
\alpha=
\begin{cases}
    \frac{w-6}{2}& \text{if } w \text{ is even,}\\
    \frac{w-9}{2}& \text{if } w \text{ is odd.}
\end{cases}
\]

We need a restriction on $p$ and $k$ to assure we have enough edges to make $\G$ connected but not too many edges to exceed maximum allowed amount of edges. The minimum number of edges that guarantee $\mathcal{G}$ is connected is $k-1$. On the other hand, the maximum amount of edges a graph can have is $\binom{k}{2}$, so
\[
k-1\leq p\leq \binom{k}{2}.
\]
From $w=2k+6+4p$, we have $p=\frac12(\alpha-k)$.
To get an upper bound on $k$, replace $p$ with $k-1$ to obtain,
\begin{align*}
k-1\leq \frac{\alpha-k}{2}&
\Longrightarrow k\leq \frac{\alpha+2}{3}.
\end{align*}
Similarly, to get a lower bound on $k$, replace $p$ with $\binom{k}{2}$ to obtain,
\begin{align*}
\binom{k}{2}=\frac{k(k-1)}{2}\geq \frac{\alpha-k}{2}
&\Longrightarrow k\geq \sqrt{\alpha}.
\end{align*}
For each integer $k$ in the interval $\left[\sqrt{\alpha},\frac{\alpha+2}{3}\right]$, we can choose a value for $p$ such that we can construct an Anosov graph. The number of integers between $\sqrt{\alpha}$ and $\frac{\alpha+2}{3}$ is given by
\[
\floor*{\frac{\alpha+2}{3}}-\ceil*{\sqrt{\alpha}}+1,
\]
where for a real number $r$, $\floor*{r}$ denotes the largest integer not greater than $r$  and   $\ceil*{r}$ denotes the smallest integer not less than $r$.
We know that the decimal part of $\frac{\alpha+2}{3}$ is at most $\frac{2}{3}$ since $\alpha$ is an integer, so $\floor*{\frac{\alpha+2}{3}} \geq \frac{\alpha+2}{3}-\frac{2}{3}$ and thus
\[
\floor*{\frac{\alpha+2}{3}}-\ceil*{\sqrt{\alpha}}+1\geq \left(\frac{\alpha+2}{3}-\frac{2}{3}\right)-(\sqrt{\alpha}+1)+1=\frac{\alpha+2}{3}-\sqrt{\alpha}-\frac{2}{3}.
\]
So the number of integer values between $\sqrt{\alpha}$ and $\frac{\alpha+2}{3}$ is at least $\frac{1}{3}(\alpha-3\sqrt{\alpha})$. Therefore,
\[
\nu(w)>\frac{1}{3}(\alpha-3\sqrt{\alpha}).
\]
We know $w-9< w-6< w$, for all $w$.
Therefore
\[
\frac{1}{3}(\alpha-3\sqrt{\alpha}) > \frac{1}{3} \left(\frac{w-9}{2}-3\sqrt{\frac{w}{2}}\right)=\frac{1}{6}(w-3\sqrt{2w}-9).
\]

We repeat the process with a different class of Anosov graphs. Suppose we set $00\in\E$ and $\wt(0)=q\geq 3$ in $\mathcal{G}$.  Then the number of edges in $G$ is $\frac{1}{2}(q^2+3q+8p)$ and the number of vertices in $G$ is $2k+q$ so $w=\frac{1}{2}(4k+q^2+5q+8p)$. Let $w>18$ be given. If $w$ is odd we choose $q=3$, otherwise we choose $q=4$. Therefore $w=2k+12+4p$ or $w=2k+18+4p$; i.e.,
\[
\frac{w-12}{2}=k+2p\hspace{2em}\text{ or }
\hspace{2em}
\frac{w-18}{2}=k+2p.
\]
Let
\[
\alpha=
\begin{cases}
    \frac{w-18}{2}& \text{if } w \text{ is even,}\\
    \frac{w-12}{2}& \text{if } w \text{ is odd.}
\end{cases}
\]
If $k\in\left[\sqrt{\alpha},\frac{\alpha+2}{3}\right]$, then any $p$ value which satisfies $\alpha=k+2p$ also satisfies $k-1\le p \le \binom{k}{2}$. Since we are working with graphs, $k$ must be an integer and the number of integer values between $\sqrt{\alpha}$ and $\frac{\alpha+2}{3}$ is at least $\frac{1}{3}(\alpha-3\sqrt{\alpha})$. Therefore,
$$\nu(w)\ge \frac{1}{3}(\alpha-3\sqrt{\alpha}).$$
We know $w-18 < w-12 < w,$ for all $w$. Therefore
\[
\frac{1}{3}(\alpha-3\sqrt{\alpha}) > \frac{1}{3}\left(\frac{w-18}{2}-3\sqrt{\frac{w}{2}}\right) = \frac{1}{6}(w-3\sqrt{2w}-18),
\]
and thus
\[
\nu(w)> \frac{1}{6}(w-3\sqrt{2w}-18).
\]
Since we counted two distinct families we can add the lower bounds from both cases together to get,
\[
\nu(w)>\frac{w}{3}-\sqrt{2w}-\frac{9}{2}.
\]
But, as we stated before $w=n+m$ this completes the proof.
\end{proof}

\section{Bounds on the number of Anosov graphs}\label{boundsAnosov}

For the following upper and lower bounds we will need sufficient conditions for a graph with a weight function to be the quotient graph of an Anosov graph in terms of its adjacency matrix. However all we need for this is an alteration of our previous result Theorem $\ref{Brick}$.
\begin{thm}
\label{brickmatrix}
A connected graph $\G=([k],\E,\wt)$ is a quotient graph $G/{\sim}$ where $G$ is Anosov of type $(\wt(1),\dots,\wt(k))$ provided,
\begin{enumerate}
\item $\wt(i)\ge 2$ for all $i\in[k]$;
\item for all $i\in[k]$ if $\wt(i)=2$ then $ii\not\in\E$;
\item no two rows in the adjacency matrix $A(\G)$ repeat.
\end{enumerate}
\end{thm}
\begin{proof}
Let $A(\G) = (a_{ij})$.  By Theorem \ref{Brick}, we need to show that having no two rows in $A(\G)$ repeat implies: if $i,j\in[k]$ are distinct, then $ii\in\E$ and $jj\in \E$ implies $\overline{N}_{\mathcal{G}}(i)\neq\overline{N}_{\mathcal{G}}(j)$.  Similarly, $ii\not\in\E$ and $jj\not\in \E$ implies $N_{\mathcal{G}}(i)\neq N_{\mathcal{G}}(j)$. Let $i\neq j$, $ii\in\E$, $jj\in \E$ and assume by way of contradiction that $\overline N_{\G}(i)=\overline N_{\G}(j)$. Since $\overline{N}(i)=\overline{N}(j)$ we have $a_{in}=a_{jn}$ for all $n\in[k]$, but this would result in two rows being identical.  Similarly, let $i\neq j$, $ii\not\in\E$, $jj\not\in \E$ and assume by way of contradiction that $N(i)=N(j)$.  Since $ii\not\in\E$ and $jj\not\in \E$, we know $a_{ii}=0$ and $a_{jj}=0$.  Since $N(i)=N(j)$, we know $a_{ij}=0$, $a_{ji}=0$, and $a_{in}=a_{jn}$ for all $n\in[k]$. This results in two rows being identical in $A(\G)$, which is a contradiction.\end{proof}

\begin{lemma} \label{countinglambda}
Let $n\in \NN$ and $\lambda\vdash n$. If $G = (V, E)$ is a graph with $|V|=\ell(\lambda)$ then the number of ways to map the parts of $\lambda$ to the vertices of $G$ is
\[
\frac{\ell(\lambda)!}{\prod_{i=1}^{\lambda_1} k(\lambda,i)!},
\]
where  $ \lambda = (\lambda_1, \dots, \lambda_{\ell(\lambda)})$ and $k(\lambda,i)=|\{\lambda_j:\lambda_j=i\}|$.
\end{lemma}
\begin{proof}
 The number of ways to order $n$ objects in which $n_i$ of  them are of type $i$ for all $i\in[r]$ is
\[
\frac{n!}{n_1!n_2!\dots n_r!},
\]
where $r$ is the number of different types. Hence, in our case we have,
\[
\frac{\ell(\lambda)!}{\prod_{i=1}^{\lambda_1} k(\lambda,i)!}. \qedhere
\]
 \end{proof}

\begin{dfn} Define $X(t)$ to be the number of $t\times t$ symmetric binary matrices with no repeating rows up to graph equivalence.
\end{dfn}
\begin{thm}  \label{The big one}
Let $a(n)$ be the number of Anosov graphs on $n$ vertices and $\Lambda_j(n)=\{\lambda : \lambda \vdash n\text{ and }\lambda_{\ell(\lambda)}\ge j\}$. Then
\[
\frac{1}{2}\sum_{\lambda\in\Lambda_3(n)} X(\ell(\lambda))\le a(n) \le \sum_{\lambda\in\Lambda_2(n)} \frac{\ell(\lambda)!}{\prod_{i=1}^{\lambda_1} k(\lambda,i)!}X(\ell(\lambda)),
\]
where $k(\lambda,i)=|\{\lambda_j:\lambda_j=i\}|$.
\end{thm}

\begin{proof} By Proposition \ref{UpandLow}, if we show
\[U(n)\le\sum_{\lambda\in\Lambda_2(n)} \frac{X(\ell(\lambda)) \ell(\lambda)!}{\prod_{i=1}^{\lambda_1} k(\lambda,i)!}\hspace{2em}\text{ and }\hspace{2em}L(n)\ge\frac{1}{2}\sum_{\lambda\in\Lambda_3(n)} X(\ell(\lambda)),\]
then we are done. Let $n\in \NN$ be given.
By Theorem \ref{brickmatrix} each matrix counted in $X(t)$ is a distinct adjacency matrix of an Anosov graph provided we attach the appropriate weights.
By Lemma \ref{countinglambda}, we know the number of graphs of type $\lambda$ is at most
\[
\frac{\ell(\lambda)!}{\prod_{i=1}^{\lambda_1} k(\lambda,i)!}X(\ell(\lambda)).
\]
Sum over all partitions of $n$ with part size greater than or equal to  2 to obtain,
\[
U(n)\leq \sum_{\lambda \in\Lambda_2(n)}\frac{\ell(\lambda)!}{\prod_{i=1}^{\lambda_1} k(\lambda,i)!}X(\ell(\lambda)).
\]
We can use a similar strategy for a lower bound on $L(n)$. Let $J(n)$ be the number of graphs of type $(\lambda_1,\dots,\lambda_k)$ where $\lambda_k\geq 3$, not necessarily connected. It can be seen that a graph is connected if the complement of the graph is disconnected.  At worst, half of the graphs we count in $J(n)$ are disconnected, so
\[
\frac{1}{2} J(n) \leq L(n).
\]
Unlike in the previous part, we cannot guarantee that rearranging weights with a adjacency matrix produces two district graphs. In fact, there are partitions that, when used as weights, produce the same graph no matter how they are arranged. Thus, the best bound we can obtain by arranging weights is,
\[
\sum_{\lambda\in\Lambda_3(n)} X(\ell(\lambda)) \le J(n),
\]
which gives us
\[
\frac{1}{2} \sum_{\lambda\in\Lambda_3(n)} X(\ell(\lambda)) \le L(n). \qedhere
\]
\end{proof}

\begin{remark}
If we omit the restriction that a graph must be connected to be Anosov, then the lower bound from Theorem \ref{The big one} changes. In the proof of Theorem \ref{The big one}, we found a lower bound on $J(n)$, the number of graphs of type $\lambda=(\lambda_1,\dots,\lambda_k)$, where $\lambda_k\geq 3$.  However, since we no longer need connectedness for a graph to be Anosov, we have that every graph of type $\lambda=(\lambda_1,\dots,\lambda_k)$, where $\lambda_k\geq 3$, is Anosov. Hence $J(n)\leq a(n)$ and
\[\sum_{\Lambda_3(n)} X(\ell(\lambda))\leq a(n).\]
\end{remark}


A major caveat of these previous proofs is that there is no explicit definition for $X(t)$. However, we can use lower and upper bounds on $X(t)$ to take its place in the proof above. Since each matrix counted in $X(t)$ is symmetric, we can use the number of $t\times t$ binary symmetric matrices as an upper bound, $2^{{t+1}\choose{2}}/t!$. We know $2^{{t+1}\choose{2}}$ is the number of graphs with $t$ vertices, so at worst there are $t!$ graphs which are isomorphic to each other for each isomorphism class.

 Similarly we can use $2^t$ as a lower bound for $X(t)$.

\begin{lemma} \label{weak lower bound}
 Given a  binary string of length $t$, we can construct a unique $t\times t$ symmetric binary matrix with no repeating rows.
\end{lemma}
\begin{proof}
If $\alpha=\{\alpha_1,\dots,\alpha_t\}$ is a binary string, then define $\overline \alpha_i$ to be the opposite of $\alpha_i$ and define a $t\times t$ matrix $A = (a_{ij})$ as follows:
\[
a_{ij}=
\left\{
\begin{matrix}
\alpha_j & \text{if }j> i,\\
\alpha_i & \text{if } i> j,\\
\overline\alpha_i & \text{if } i=j.
\end{matrix}\right.
\]
For each binary string $\alpha$ of length $t$, we get a unique matrix, since the first row is $(\overline\alpha_1,\alpha_2,\dots,\alpha_t)$.  It remains to show that no two rows of $A$ repeat. Let $i,j\in[t]$ be given. Without loss of generality, assume $1<i<j$. We want to compare the two strings
\[
(a_{i1},\dots,a_{ij},\dots,a_{it})
\qquad \text{ and } \qquad
(a_{j1}\dots,a_{jj},\dots,a_{jt}).
\]
Since $j>i$, we have $a_{ij}=\alpha_j$, but $a_{jj}=\overline \alpha_j$, so the two strings cannot be the same. Since these two rows were chosen arbitrarily, we have shown $A$ has no repeating rows.
\end{proof}

From Lemma \ref{weak lower bound}, we obtain $\frac{2^t}{t!}\le X(t)$.  We need $\frac{2^t}{t!}$ and not just $2^t$ since there is no guarantee that two distinct matrices are not adjacency matrices of the same graph.  However, there are at most $t!$ matrices that are adjacency matrices of the same graph.

We can generalize the strategy of comparing rows to create an even stronger lower bound.

\begin{thm}
For each $t\in \NN$, \[2^{t-1}(2^t - 1)\prod_{i=2}^{t-1}\max\left( 2^{t-i}-i,1\right)\le X(t) t!\]
\end{thm}
\begin{proof}
We will construct a $t\times t$ symmetric binary matrix $A = (a_{ij})$ such that no two rows repeat.  We will count the number of ways to fill the $i$th row that the entries in rows $1,\dots,i-1$ are determined.  Every entry of the first row is either $0$ or $1$, so the number of possible first rows in $A$ is $2^t$.  Since $A$ is symmetric $a_{12}=a_{21}$.  Thus there are either $2^{t-1}$ choices or $2^{t-1}-1$ choices for the second row depending on whether or not $a_{11}=a_{21}$.  Half of the time there will be $2^{t-1}$ choices for the second row, and for the other half, there will be $2^{t-1}-1$ choices for the second row.  Hence the number of combinations for the first two rows is exactly,
\[\frac{2^t}{2}(2^{t-1})+\frac{2^t}{2}\left(2^{t-1}-1\right).\]
The tree below illustrates the possible branching paths when we choose entries for the third row.
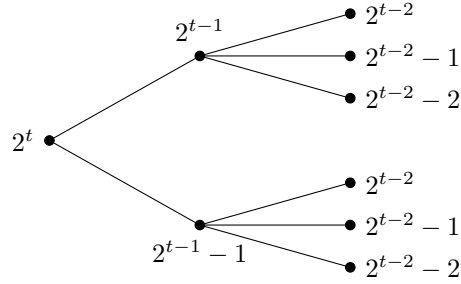
\[\begin{tikzpicture}[font=\small,baseline=-4,yscale=.75]
\node[vertex,label={left:$2^t$}] (1) at (0,0) {};
\node[vertex,label={below:$2^{t-1}-1$}] (2) at (2,-1.5) {};
\node[vertex,label={above:$2^{t-1}$}] (3) at (2,1.5) {};

\node[vertex,label={right:$2^{t-2}-1$}] (4) at (4,-1.5) {};
\node[vertex,label={right:$2^{t-2} $}] (5) at (4,-.75) {};
\node[vertex,label={right:$2^{t-2}-2 $}] (8) at (4,-2.25) {};

\node[vertex,label={right:$2^{t-2}-2$}] (6) at (4,.75) {};
\node[vertex,label={right:$2^{t-2} -1$}] (7) at (4,1.5) {};
\node[vertex,label={right:$2^{t-2} $}] (9) at (4,2.25) {};
\path[-]
 (1) edge (2)
 (1) edge (3)
 (2) edge (4)
 (2) edge (5)
 (2) edge (8)
 (3) edge (6)
 (3) edge (7)
 (3) edge (9);
\end{tikzpicture}\]
A lower bound for $X(t)$ is obtained by always choosing the bottom path on the graph. If the graph were extended to the $t$th step and we follow the lowest path, we would have exactly
\[2^t\prod_{i=1}^{t-1}\max\left( 2^{t-i}-i,1\right).\]
However, we know the probability of the first choice is cut evenly so
\[
 2^{t-1}(2^{t} - 1) \prod_{i=2}^{t-1}\max\left( 2^{t-i}-i,1\right)\le X(t)t!.\qedhere
\]
\end{proof}

%

\section{Bounded by $p(n)$}\label{bound p(n)}
In addition to finding explicit bounds on $a(n)$, we can show that $a(n)$ grows exponentially as $n$ increases. We will do this by showing there exists an injective function from $\Lambda(n)$ to $\mathcal{A}_n$ after some small value for $n$ using quotient graphs, where $\Lambda(n)$ is the set of partitions of $n$.  The fact that $a(n)$ grows exponentially then follows from the well-known result by Hardy and Ramanujan \cite{HR} that $p(n) \sim (4n\sqrt3)^{-1}\exp(\pi\sqrt{2n/3})$.
\begin{thm}\label{p(n)a(n)}
 Let $n\in \NN$ with $n\geq 9$. If $p(n)$ is the number of partitions on $n$ and $a(n)$ is the number of Anosov graphs on $n$ vertices then,
\[p(n)\leq a(n).\]
\end{thm}
\begin{proof}
Let $n\geq 9$. The deconstruction function, $\mathcal{D}$, restricted to quotient graph $G/{\sim}$ is a bijection into graphs (see Section \ref{brick}).
If $\mathcal{D}_n^*$ is the restriction of $\mathcal{D}$ to quotient graphs induced by $\sim$ of Anosov graphs on $n$ vertices, then $\mathcal{D}_n^*$ is a bijection from quotient graphs induced by $\sim$ of Anosov graphs on $n$ vertices to $\mathcal{A}_n$.
It suffices to show there is an injective function $B$ from $\Lambda(n)$ to quotient graphs induced by $\sim$ of Anosov graphs on $n$ vertices.

Let $\lambda\vdash n$. We will define $B$ explicitly through several cases. Assume $\lambda_i \geq 2$ for all $i$.
\begin{description}
\item[Case 1] $\lambda = (\lambda_1)$
\[B(\lambda)=
\begin{tikzpicture}[font=\small,baseline=-4]
\node[cl,label={below:$\lambda_1$}] (1) at (0,0) {};
\path[-]
 (1) edge [loopup] (1);
\end{tikzpicture}
\]

\item[Case 2] $\lambda = (\lambda_1,\lambda_2)$
\[B(\lambda)=
\begin{tikzpicture}[font=\small,baseline=-4]
\node[op,label={below:$\lambda_1$}] (1) at (0,0) {};
\node[op,label={below:$\lambda_2$}] (2) at (1,0) {};
\path[-]
 (1) edge (2);
\end{tikzpicture}
\]

\item[Case 3]  $\lambda = (\lambda_1,\lambda_2,\lambda_3)$
\[B(\lambda)=
\begin{tikzpicture}[font=\small,baseline=-4]
\node[op,label={below:$\lambda_1$}] (1) at (0,0) {};
\node[op,label={below:$\lambda_2$}] (2) at (1,0) {};
\node[op,label={below:$\lambda_3$}] (3) at (2,0) {};
\path[-]
 (1) edge (2)
 (2) edge (3)
 (1) edge [bend left] node[left] {} (3);
\end{tikzpicture}
\]

\item[Case 4] $\lambda = (\lambda_1,\dots,\lambda_k); k \geq 4$
\[B(\lambda)=
\begin{tikzpicture}[font=\small,baseline=-4]
\node[op,label={below:$\lambda_1$}] (1) at (0,0) {};
\node[op,label={below:$\lambda_2$}] (2) at (1,0) {};
\node[op,label={below:$\lambda_4$}] (4) at (3,0) {};
\node[] (3) at (1.5,0) {};
\node[] (5) at (2.5,0) {};
\node[label={$\dots$}] (11) at (2,-.30) {};
\path[-]
 (1) edge (2)
 (2) edge (3)
 (5) edge (4);
\end{tikzpicture}
\]

\item[Case 5] $\lambda = (\lambda_1,1)$
\[B(\lambda)=
\begin{tikzpicture}[font=\small,baseline=-4,xscale=1.5]
\node[op,label={below:$2$}] (1) at (0,0) {};
\node[cl,label={below:$\lambda_1-5$}] (2) at (1,0) {};
\node[op,label={below:$2$}] (3) at (2,0) {};
\node[op,label={below:$2$}] (4) at (3,0) {};
\path[-]
 (2) edge [loopup] (2)
 (1) edge (2)
 (1) edge [bend left] node[left] {} (3)
 (2) edge (3)
 (3) edge (4)
 (2) edge [bend left] node[left] {} (4)
 (1) edge [bend left] node[left] {} (4);
\end{tikzpicture}
\]

\item[Case 6] $\lambda = (\lambda_1,\lambda_2,1)$
\[B(\lambda)=
\begin{tikzpicture}[font=\small,baseline=-4,xscale=1.5]
\node[cl,label={below:$\lambda_1-1$}] (1) at (0,0) {};
\node[op,label={below:$\lambda_2$}] (2) at (1,0) {};
\node[op,label={below:$2$}] (3) at (2,0) {};
\path[-]
 (1) edge [loopup] (1)
 (1) edge (2)
 (1) edge [bend left] node[left] {} (3)
 (2) edge (3);
\end{tikzpicture}
\]

\item[Case 7] $\lambda = (\lambda_1,\lambda_2,\lambda_3,1)$
\[B(\lambda)=
\begin{tikzpicture}[font=\small,baseline=-4,xscale=1.5]
\node[cl,label={below:$\lambda_1$}] (1) at (0,0) {};
\node[op,label={below:$\lambda_3$}] (2) at (1,0) {};
\node[cl,label={below:$\lambda_2+1$}] (3) at (2,0) {};
\path[-]
 (1) edge [loopup] (1)
 (3) edge [loopup] (3)
 (1) edge (2)
 (2) edge (3);
\end{tikzpicture}
\]

\item[Case 8] $\lambda = (\lambda_1,\lambda_2,\lambda_3,\lambda_4,1)$
\[B(\lambda)=
\begin{tikzpicture}[font=\small,baseline=-4,xscale=1.5]
\node[cl,label={below:$\lambda_1+1$}] (1) at (0,0) {};
\node[op,label={below:$\lambda_2$}] (2) at (1,0) {};
\node[op,label={below:$\lambda_3$}] (3) at (2,0) {};
\node[op,label={below:$\lambda_4$}] (4) at (3,0) {};
\path[-]
 (1) edge [loopup] (1)
 (1) edge (2)
 (1) edge [bend left] node[left] {} (3)
 (2) edge [bend left] node[left] {} (4)
 (3) edge (4)
 (2) edge (3);
\end{tikzpicture}
\]

\item[Case 9] $\lambda = (\lambda_1,\dots,\lambda_k,1); k\geq 5$
\[B(\lambda)=
\begin{tikzpicture}[font=\small,baseline=-4,xscale=1.25]
\node[op,label={below:$\lambda_k$}] (1) at (0,0) {};
\node[cl,label={below:$\lambda_1+1$}] (2) at (1,0) {};
\node[op,label={below:$\lambda_2$}] (3) at (2,0) {};
\node[op,label={below:$\lambda_3$}] (4) at (3,0) {};
\node[op,label={below:$\lambda_k-1$}] (5) at (5,0) {};
\node[] (13) at (3.5,0) {};
\node[] (15) at (4.5,0) {};
\node[label={$\dots$}] (11) at (4,-.30) {};
\path[-]
 (2) edge [loopup] (2)
 (1) edge (2)
 (1) edge [bend left] node[left] {} (3)
 (2) edge (3)
 (3) edge (4)
 (4) edge (13)
 (5) edge (15)
 (1) edge [bend left] node[left] {} (5);
\end{tikzpicture}
\]

\item[Case 10] $\lambda = (\lambda_1,\dots,\lambda_k,1^t); \lambda_1\geq 3,t\geq 2$
\[B(\lambda)=
\begin{tikzpicture}[font=\small,baseline=-4,xscale=1.5]
\node[cl,label={below:$\lambda_1$}] (1) at (0,0) {};
\node[op,label={below:$\lambda_2$}] (2) at (1,0) {};
\node[op,label={below:$\lambda_k$}] (3) at (3,0) {};
\node[op,label={below:$t$}] (4) at (4,0) {};
\node[] (13) at (1.5,0) {};
\node[] (15) at (2.5,0) {};
\node[label={$\dots$}] (11) at (2,-.30) {};
\path[-]
 (1) edge [loopup] (1)
 (1) edge (2)
 (3) edge (4)
 (2) edge (13)
 (3) edge (15);
\end{tikzpicture}
\]

\item[Case 11] $\lambda = (2,1^t);t\geq 7$
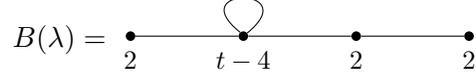
\[B(\lambda)=
\begin{tikzpicture}[font=\small,baseline=-4,xscale=1.5]
\node[op,label={below:$2$}] (1) at (0,0) {};
\node[cl,label={below:$t-4$}] (2) at (1,0) {};
\node[op,label={below:$2$}] (3) at (2,0) {};
\node[op,label={below:$2$}] (4) at (3,0) {};
\path[-]
 (2) edge [loopup] (2)
 (1) edge (2)
 (2) edge (3)
 (3) edge (4);
\end{tikzpicture}
\]

\item[Case 12] $\lambda = (2^2,1^t);t\geq 5$
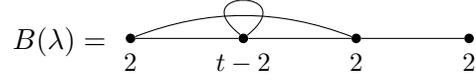
\[B(\lambda)=
\begin{tikzpicture}[font=\small,baseline=-4,xscale=1.5]
\node[op,label={below:$2$}] (1) at (0,0) {};
\node[cl,label={below:$t-2$}] (2) at (1,0) {};
\node[op,label={below:$2$}] (3) at (2,0) {};
\node[op,label={below:$2$}] (4) at (3,0) {};
\path[-]
 (2) edge [loopup] (2)
 (1) edge (2)
 (1) edge [bend left] node[left] {} (3)
 (2) edge (3)
 (3) edge (4);
\end{tikzpicture}
\]

\item[Case 13] $\lambda = (2^k,1^t);t\geq 2$
\[B(\lambda)=K^{k+1},\] where all vertices have weight 2 and are without loops except for one vertex which has weight $t$.

\item[Case 14] $\lambda = (1^t);t\geq 9$
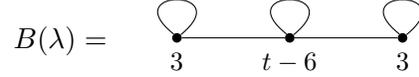
\[B(\lambda)=
\begin{tikzpicture}[font=\small,baseline=-4,xscale=1.5]
\node[cl,label={below:$3$}] (1) at (0,0) {};
\node[cl,label={below:$t-6$}] (2) at (1,0) {};
\node[cl,label={below:$3$}] (3) at (2,0) {};
\path[-]
 (1) edge [loopup] (1)
 (2) edge [loopup] (2)
 (3) edge [loopup] (3)
 (1) edge (2)
 (2) edge (3);
\end{tikzpicture}
\]
\end{description}
Since each case produces a distinct quotient graph induced by $\sim$ of an Anosov graph, we have $B$ is injective. Therefore $\mathcal{D}^*_n\circ B$ is an injective function from $\Lambda(n)$ to $\mathcal{A}_n$.
\end{proof}

%
%
%
%

\appendix

\section{Tables of Anosov graphs with at most 8 vertices}\label{appA}

\begin{longtable}{| p{.4\textwidth} | p{.4\textwidth} |}
\hline
\center{
\begin{tikzpicture}[scale = .4, font=\small,baseline=-10]
\node[vertex] (1) at (0,0) {};
\node[vertex] (2) at (2,0) {};
\node[vertex] (3) at (1,1) {};
\path[-]
 (1) edge (2)
 (1) edge (3)
 (2) edge (3);
\end{tikzpicture}
}
&
\center{
\begin{tikzpicture}[font=\small,baseline=-20]
\node[cl,label={right:$3$}] (1) at (0,-.5) {};
\path[-]
 (1) edge [loopup] (1);
\end{tikzpicture}
}\endline
\hline
\caption{Anosov graphs on three vertices and their corresponding quotient graphs induced by $\sim$.}\label{anosov_list_3}
\end{longtable}

\begin{longtable}{| p{.4\textwidth} | p{.4\textwidth} |}
\hline
\center{
\begin{tikzpicture}[scale = .4, font=\small,baseline=-10]
\node[vertex] (1) at (0,0) {};
\node[vertex] (2) at (2,0) {};
\node[vertex] (3) at (0,2) {};
\node[vertex] (4) at (2,2) {};
\path[-]
 (1) edge (2)
 (1) edge (3)
 (1) edge (4)
 (2) edge (3)
 (2) edge (4)
 (3) edge (4);
\end{tikzpicture}
}
&
\center{
\begin{tikzpicture}[font=\small,baseline=-25]
\node[cl,label={right:$4$}] (1) at (0,-1) {};
\path[-]
 (1) edge [loopup] (1);
\end{tikzpicture}
}\endline
\hline
\center{
\begin{tikzpicture}[scale = .4, font=\small,baseline=-10]
\node[vertex] (1) at (0,0) {};
\node[vertex] (2) at (2,0) {};
\node[vertex] (3) at (0,2) {};
\node[vertex] (4) at (2,2) {};
\path[-]
 (1) edge (2)
 (1) edge (4)
 (2) edge (3)
 (3) edge (4);
\end{tikzpicture}
}
&
\center{
\begin{tikzpicture}[font=\small,baseline=-30]
\node[op,label={below:$2$}] (1) at (0,-1) {};
\node[op,label={below:$2$}] (2) at (1,-1) {};
\path[-]
 (1) edge (2);
\end{tikzpicture}
}\endline
\hline
\caption{Anosov graphs on four vertices and their corresponding quotient graphs induced by $\sim$.}\label{anosov_list_4}
\end{longtable}

\begin{longtable}{| p{.4\textwidth} | p{.4\textwidth} |}
\hline
\center{
\begin{tikzpicture}[scale = .4, font=\small,baseline=-10]
\node[vertex] (1) at (0,1) {};
\node[vertex] (2) at (2,0) {};
\node[vertex] (3) at (4,1) {};
\node[vertex] (4) at (1,2) {};
\node[vertex] (5) at (3,2) {};
\path[-]
 (1) edge (2)
 (1) edge (3)
 (1) edge (4)
 (1) edge (5)
 (2) edge (3)
 (2) edge (4)
 (2) edge (5)
 (3) edge (4)
 (3) edge (5)
 (4) edge (5);
\end{tikzpicture}
}
&
\center{
\begin{tikzpicture}[font=\small,baseline=-25]
\node[cl,label={right:$5$}] (1) at (0,-1) {};
\path[-]
 (1) edge [loopup] (1);
\end{tikzpicture}
}\endline
\hline
\center{
\begin{tikzpicture}[scale = .4, font=\small,baseline=-10]
\node[vertex] (1) at (0,1) {};
\node[vertex] (2) at (2,0) {};
\node[vertex] (3) at (4,1) {};
\node[vertex] (4) at (1,2) {};
\node[vertex] (5) at (3,2) {};
\path[-]
 (1) edge (2)
 (1) edge (3)
 (1) edge (4)
 (1) edge (5)
 (2) edge (3)
 (2) edge (4)
 (2) edge (5)
 (3) edge (4)
 (3) edge (5);
\end{tikzpicture}
}
&
\center{
\begin{tikzpicture}[font=\small,baseline=-30]
\node[cl,label={below:$3$}] (1) at (0,-1) {};
\node[op,label={below:$2$}] (2) at (1,-1) {};
\path[-]
 (1) edge [loopup] (1)
 (1) edge (2);
\end{tikzpicture}
}\endline
\hline
\center{
\begin{tikzpicture}[scale = .4, font=\small,baseline=-10]
\node[vertex] (1) at (0,1) {};
\node[vertex] (2) at (2,0) {};
\node[vertex] (3) at (4,1) {};
\node[vertex] (4) at (1,2) {};
\node[vertex] (5) at (3,2) {};
\path[-]
 (1) edge (4)
 (1) edge (5)
 (2) edge (4)
 (2) edge (5)
 (3) edge (4)
 (3) edge (5);
\end{tikzpicture}
}
&
\center{
\begin{tikzpicture}[font=\small,baseline=-30]
\node[op,label={below:$3$}] (1) at (0,-1) {};
\node[op,label={below:$2$}] (2) at (1,-1) {};
\path[-]
 (1) edge (2);
\end{tikzpicture}
}\endline
\hline
\caption{Anosov graphs on five vertices and their corresponding quotient graphs induced by $\sim$.}\label{anosov_list_5}
\end{longtable}

\begin{longtable}{| p{.4\textwidth} | p{.4\textwidth} |}
\hline
\center{
\begin{tikzpicture}[scale = .4,  font=\small,baseline=-10]
\node[vertex] (1) at (0,1) {};
\node[vertex] (2) at (2,0) {};
\node[vertex] (3) at (4,1) {};
\node[vertex] (4) at (0,2) {};
\node[vertex] (5) at (2,3) {};
\node[vertex] (6) at (4,2) {};
\path[-]
 (1) edge (2)
 (1) edge (3)
 (1) edge (4)
 (1) edge (5)
 (1) edge (6)
 (2) edge (3)
 (2) edge (4)
 (2) edge (5)
 (2) edge (6)
 (3) edge (4)
 (3) edge (5)
 (3) edge (6)
 (4) edge (5)
 (4) edge (6)
 (5) edge (6);
\end{tikzpicture}
}
&
\center{
\begin{tikzpicture}[font=\small,baseline=-25]
\node[cl,label={right:$6$}] (1) at (0,-1) {};
\path[-]
 (1) edge [loopup] (1);
\end{tikzpicture}
}\endline
\hline
\center{
\begin{tikzpicture}[scale = .4,  font=\small,baseline=-10]
\node[vertex] (1) at (0,1) {};
\node[vertex] (2) at (2,0) {};
\node[vertex] (3) at (4,1) {};
\node[vertex] (4) at (0,2) {};
\node[vertex] (5) at (2,3) {};
\node[vertex] (6) at (4,2) {};
\path[-]
 (1) edge (2)
 (1) edge (3)
 (1) edge (4)
 (1) edge (5)
 (1) edge (6)
 (2) edge (3)
 (2) edge (4)
 (2) edge (5)
 (2) edge (6)
 (3) edge (4)
 (3) edge (5)
 (3) edge (6)
 (4) edge (5)
 (4) edge (6);
\end{tikzpicture}
}
&
\center{
\begin{tikzpicture}[font=\small,baseline=-25]
\node[cl,label={below:$4$}] (1) at (0,-1) {};
\node[op,label={below:$2$}] (2) at (1,-1) {};
\path[-]
 (1) edge [loopup] (1)
 (1) edge (2);
\end{tikzpicture}
}\endline
\hline
\center{
\begin{tikzpicture}[scale = .4,  font=\small,baseline=-10]
\node[vertex] (1) at (0,1) {};
\node[vertex] (2) at (2,0) {};
\node[vertex] (3) at (4,1) {};
\node[vertex] (4) at (0,2) {};
\node[vertex] (5) at (2,3) {};
\node[vertex] (6) at (4,2) {};
\path[-]
 (1) edge (5)
 (1) edge (6)
 (2) edge (5)
 (2) edge (6)
 (3) edge (5)
 (3) edge (6)
 (4) edge (5)
 (4) edge (6);
\end{tikzpicture}
}
&
\center{
\begin{tikzpicture}[font=\small,baseline=-25]
\node[op,label={below:$4$}] (1) at (0,-1) {};
\node[op,label={below:$2$}] (2) at (1,-1) {};
\path[-]
 (1) edge (2);
\end{tikzpicture}
}\endline
\hline
\center{
\begin{tikzpicture}[scale = .4,  font=\small,baseline=-10]
\node[vertex] (1) at (0,1) {};
\node[vertex] (2) at (2,0) {};
\node[vertex] (3) at (4,1) {};
\node[vertex] (4) at (0,2) {};
\node[vertex] (5) at (2,3) {};
\node[vertex] (6) at (4,2) {};
\path[-]
 (1) edge (4)
 (1) edge (5)
 (1) edge (6)
 (2) edge (4)
 (2) edge (5)
 (2) edge (6)
 (3) edge (4)
 (3) edge (5)
 (3) edge (6)
 (4) edge (5)
 (4) edge (6)
 (5) edge (6);
\end{tikzpicture}
}
&
\center{
\begin{tikzpicture}[font=\small,baseline=-25]
\node[cl,label={below:$3$}] (1) at (0,-1) {};
\node[op,label={below:$3$}] (2) at (1,-1) {};
\path[-]
 (1) edge [loopup] (1)
 (1) edge (2);
\end{tikzpicture}
}\endline
\hline
\center{
\begin{tikzpicture}[scale = .4,  font=\small,baseline=-10]
\node[vertex] (1) at (0,1) {};
\node[vertex] (2) at (2,0) {};
\node[vertex] (3) at (4,1) {};
\node[vertex] (4) at (0,2) {};
\node[vertex] (5) at (2,3) {};
\node[vertex] (6) at (4,2) {};
\path[-]
 (1) edge (4)
 (1) edge (5)
 (1) edge (6)
 (2) edge (4)
 (2) edge (5)
 (2) edge (6)
 (3) edge (4)
 (3) edge (5)
 (3) edge (6);
\end{tikzpicture}
}
&
\center{
\begin{tikzpicture}[font=\small,baseline=-25]
\node[op,label={below:$3$}] (1) at (0,-1) {};
\node[op,label={below:$3$}] (2) at (1,-1) {};
\path[-]
 (1) edge (2);
\end{tikzpicture}
}\endline
\hline
\center{
\begin{tikzpicture}[scale = .4,  font=\small,baseline=-10]
\node[vertex] (1) at (0,1) {};
\node[vertex] (2) at (2,0) {};
\node[vertex] (3) at (4,1) {};
\node[vertex] (4) at (0,2) {};
\node[vertex] (5) at (2,3) {};
\node[vertex] (6) at (4,2) {};
\path[-]
 (1) edge (3)
 (1) edge (4)
 (1) edge (5)
 (1) edge (6)
 (2) edge (3)
 (2) edge (4)
 (2) edge (5)
 (2) edge (6)
 (3) edge (4)
 (3) edge (5)
 (4) edge (6)
 (5) edge (6);
\end{tikzpicture}
}
&
\center{
\begin{tikzpicture}[font=\small,baseline=-10]
\node[op,label={below:$2$}] (1) at (0,-1) {};
\node[op,label={below:$2$}] (2) at (1,0) {};
\node[op,label={below:$2$}] (3) at (2,-1) {};
\path[-]
 (1) edge (2)
 (1) edge (3)
 (2) edge (3);
\end{tikzpicture}
}\endline
\hline
\caption{Anosov graphs on six vertices and their corresponding quotient graphs induced by $\sim$.}\label{anosov_list_6}
\end{longtable}

\begin{longtable}{| p{.4\textwidth} | p{.4\textwidth} |}
\hline
\center{
\begin{tikzpicture}[scale = .4,baseline=-10]
\node[vertex] (1) at (0,1) {};
\node[vertex] (2) at (2,0) {};
\node[vertex] (3) at (4,0) {};
\node[vertex] (4) at (6,1) {};
\node[vertex] (5) at (1,2) {};
\node[vertex] (6) at (3,3) {};
\node[vertex] (7) at (5,2) {};
\path[-]
 (1) edge (2)
 (1) edge (3)
 (1) edge (4)
 (1) edge (5)
 (1) edge (6)
 (1) edge (7)
 (2) edge (3)
 (2) edge (4)
 (2) edge (5)
 (2) edge (6)
 (2) edge (7)
 (3) edge (4)
 (3) edge (5)
 (3) edge (6)
 (3) edge (7)
 (4) edge (5)
 (4) edge (6)
 (4) edge (7)
 (5) edge (6)
 (5) edge (7)
 (6) edge (7);
\end{tikzpicture}
}
&
\center{
\begin{tikzpicture}[font=\small,baseline=-20]
\node[cl,label={right:$7$}] (1) at (0,-1) {};
\path[-]
 (1) edge [loopup] (1);
\end{tikzpicture}
}\endline
\hline
\center{
\begin{tikzpicture}[scale = .4,baseline=-10]
\node[vertex] (1) at (0,1) {};
\node[vertex] (2) at (2,0) {};
\node[vertex] (3) at (4,0) {};
\node[vertex] (4) at (6,1) {};
\node[vertex] (5) at (1,2) {};
\node[vertex] (6) at (3,3) {};
\node[vertex] (7) at (5,2) {};
\path[-]
 (1) edge (2)
 (1) edge (3)
 (1) edge (4)
 (1) edge (5)
 (1) edge (6)
 (1) edge (7)
 (2) edge (4)
 (2) edge (5)
 (2) edge (6)
 (2) edge (7)
 (3) edge (4)
 (3) edge (5)
 (3) edge (6)
 (3) edge (7)
 (4) edge (5)
 (4) edge (6)
 (4) edge (7)
 (5) edge (6)
 (5) edge (7)
 (6) edge (7);
\end{tikzpicture}
}
&
\center{
\begin{tikzpicture}[font=\small,baseline=0]
\node[cl,label={below:$5$}] (1) at (0,0) {};
\node[op,label={below:$2$}] (2) at (1,0) {};
\path[-]
 (1) edge [loopup] (1)
 (1) edge (2);
\end{tikzpicture}
}\endline
\hline
\center{
\begin{tikzpicture}[scale = .4,baseline=-10]
\node[vertex] (1) at (0,1) {};
\node[vertex] (2) at (2,0) {};
\node[vertex] (3) at (4,0) {};
\node[vertex] (4) at (6,1) {};
\node[vertex] (5) at (1,2) {};
\node[vertex] (6) at (3,3) {};
\node[vertex] (7) at (5,2) {};
\path[-]
 (1) edge (2)
 (1) edge (3)
 (2) edge (4)
 (2) edge (5)
 (2) edge (6)
 (2) edge (7)
 (3) edge (4)
 (3) edge (5)
 (3) edge (6)
 (3) edge (7);
\end{tikzpicture}
}
&
\center{
\begin{tikzpicture}[font=\small,baseline=0]
\node[op,label={below:$5$}] (1) at (0,0) {};
\node[op,label={below:$2$}] (2) at (1,0) {};
\path[-]
 (1) edge (2);
\end{tikzpicture}
}\endline
\hline
\center{
\begin{tikzpicture}[scale = .4,baseline=-10]
\node[vertex] (1) at (0,1) {};
\node[vertex] (2) at (2,0) {};
\node[vertex] (3) at (4,0) {};
\node[vertex] (4) at (6,1) {};
\node[vertex] (5) at (1,2) {};
\node[vertex] (6) at (3,3) {};
\node[vertex] (7) at (5,2) {};
\path[-]
 (1) edge (2)
 (1) edge (3)
 (1) edge (4)
 (1) edge (5)
 (1) edge (6)
 (1) edge (7)
 (2) edge (3)
 (2) edge (4)
 (2) edge (5)
 (2) edge (6)
 (2) edge (7)
 (3) edge (4)
 (3) edge (5)
 (3) edge (6)
 (3) edge (7)
 (4) edge (5)
 (4) edge (6)
 (4) edge (7);
\end{tikzpicture}
}
&
\center{
\begin{tikzpicture}[font=\small,baseline=0]
\node[cl,label={below:$4$}] (1) at (0,0) {};
\node[op,label={below:$3$}] (2) at (1,0) {};
\path[-]
 (1) edge [loopup] (1)
 (1) edge (2);
\end{tikzpicture}
}\endline
\hline
\center{
\begin{tikzpicture}[scale = .4,baseline=-10]
\node[vertex] (1) at (0,1) {};
\node[vertex] (2) at (2,0) {};
\node[vertex] (3) at (4,0) {};
\node[vertex] (4) at (6,1) {};
\node[vertex] (5) at (1,2) {};
\node[vertex] (6) at (3,3) {};
\node[vertex] (7) at (5,2) {};
\path[-]
 (1) edge (5)
 (1) edge (6)
 (1) edge (7)
 (2) edge (5)
 (2) edge (6)
 (2) edge (7)
 (3) edge (5)
 (3) edge (6)
 (3) edge (7)
 (4) edge (5)
 (4) edge (6)
 (4) edge (7)
 (5) edge (6)
 (5) edge (7)
 (6) edge (7);
\end{tikzpicture}
}
&
\center{
\begin{tikzpicture}[font=\small,baseline=0]
\node[op,label={below:$4$}] (1) at (0,0) {};
\node[cl,label={below:$3$}] (2) at (1,0) {};
\path[-]
 (2) edge [loopup] (2)
 (1) edge (2);
\end{tikzpicture}
}\endline
\hline
\center{
\begin{tikzpicture}[scale = .4,baseline=-10]
\node[vertex] (1) at (0,1) {};
\node[vertex] (2) at (2,0) {};
\node[vertex] (3) at (4,0) {};
\node[vertex] (4) at (6,1) {};
\node[vertex] (5) at (1,2) {};
\node[vertex] (6) at (3,3) {};
\node[vertex] (7) at (5,2) {};
\path[-]
 (1) edge (5)
 (1) edge (6)
 (1) edge (7)
 (2) edge (5)
 (2) edge (6)
 (2) edge (7)
 (3) edge (5)
 (3) edge (6)
 (3) edge (7)
 (4) edge (5)
 (4) edge (6)
 (4) edge (7);
\end{tikzpicture}
}
&
\center{
\begin{tikzpicture}[font=\small,baseline=0]
\node[op,label={below:$4$}] (1) at (0,0) {};
\node[op,label={below:$3$}] (2) at (1,0) {};
\path[-]
 (1) edge (2);
\end{tikzpicture}
}\endline
\hline
\center{
\begin{tikzpicture}[scale = .4,baseline=-10]
\node[vertex] (1) at (0,1) {};
\node[vertex] (2) at (2,0) {};
\node[vertex] (3) at (4,0) {};
\node[vertex] (4) at (6,1) {};
\node[vertex] (5) at (1,2) {};
\node[vertex] (6) at (3,3) {};
\node[vertex] (7) at (5,2) {};
\path[-]
 (1) edge (3)
 (1) edge (4)
 (1) edge (5)
 (1) edge (6)
 (1) edge (7)
 (2) edge (3)
 (2) edge (4)
 (2) edge (5)
 (2) edge (6)
 (2) edge (7)
 (3) edge (5)
 (3) edge (6)
 (3) edge (7)
 (4) edge (5)
 (4) edge (6)
 (4) edge (7)
 (5) edge (6)
 (5) edge (7)
 (6) edge (7);
\end{tikzpicture}
}
&
\center{
\begin{tikzpicture}[font=\small,baseline=-10]
\node[op,label={below:$2$}] (1) at (0,-1) {};
\node[cl,label={below:$3$}] (2) at (1,0) {};
\node[op,label={below:$2$}] (3) at (2,-1) {};
\path[-]
 (2) edge [loopup] (2)
 (1) edge (2)
 (1) edge (3)
 (2) edge (3);
\end{tikzpicture}
}\endline
\hline
\center{
\begin{tikzpicture}[scale = .4,baseline=-10]
\node[vertex] (1) at (0,1) {};
\node[vertex] (2) at (2,0) {};
\node[vertex] (3) at (4,0) {};
\node[vertex] (4) at (6,1) {};
\node[vertex] (5) at (1,2) {};
\node[vertex] (6) at (3,3) {};
\node[vertex] (7) at (5,2) {};
\path[-]
 (1) edge (3)
 (1) edge (4)
 (1) edge (5)
 (1) edge (6)
 (1) edge (7)
 (2) edge (3)
 (2) edge (4)
 (2) edge (5)
 (2) edge (6)
 (2) edge (7)
 (3) edge (5)
 (3) edge (6)
 (3) edge (7)
 (4) edge (5)
 (4) edge (6)
 (4) edge (7);
\end{tikzpicture}
}
&
\center{
\begin{tikzpicture}[font=\small,baseline=-10]
\node[op,label={below:$2$}] (1) at (0,-1) {};
\node[op,label={below:$3$}] (2) at (1,0) {};
\node[op,label={below:$2$}] (3) at (2,-1) {};
\path[-]
 (1) edge (2)
 (1) edge (3)
 (2) edge (3);
\end{tikzpicture}
}\endline
\hline
\center{
\begin{tikzpicture}[scale = .4,baseline=-10]
\node[vertex] (1) at (0,1) {};
\node[vertex] (2) at (2,0) {};
\node[vertex] (3) at (4,0) {};
\node[vertex] (4) at (6,1) {};
\node[vertex] (5) at (1,2) {};
\node[vertex] (6) at (3,3) {};
\node[vertex] (7) at (5,2) {};
\path[-]
 (1) edge (3)
 (1) edge (4)
 (1) edge (5)
 (1) edge (6)
 (1) edge (7)
 (2) edge (3)
 (2) edge (4)
 (2) edge (5)
 (2) edge (6)
 (2) edge (7)
 (5) edge (6)
 (5) edge (7)
 (6) edge (7);
\end{tikzpicture}
}
&
\center{
\begin{tikzpicture}[font=\small,baseline=-25]
\node[cl,label={below:$3$}] (1) at (0,-1) {};
\node[op,label={below:$2$}] (2) at (1,-1) {};
\node[op,label={below:$2$}] (3) at (2,-1) {};
\path[-]
 (1) edge [loopup] (1)
 (1) edge (2)
 (2) edge (3);
\end{tikzpicture}
}\endline
\hline
\caption{Anosov graphs on seven vertices and their corresponding quotient graphs induced by $\sim$.}\label{anosov_list_7}
\end{longtable}

\begin{longtable}{| p{.4\textwidth} | p{.4\textwidth} |}
\hline
\center{
\begin{tikzpicture}[scale = .4,baseline=-10]
\node[vertex] (1) at (0,1) {};
\node[vertex] (2) at (2,0) {};
\node[vertex] (3) at (4,0) {};
\node[vertex] (4) at (6,1) {};
\node[vertex] (5) at (0,2) {};
\node[vertex] (6) at (2,3) {};
\node[vertex] (7) at (4,3) {};
\node[vertex] (8) at (6,2) {};
\path[-]
 (1) edge (2)
 (1) edge (3)
 (1) edge (4)
 (1) edge (5)
 (1) edge (6)
 (1) edge (7)
 (1) edge (8)
 (2) edge (3)
 (2) edge (4)
 (2) edge (5)
 (2) edge (6)
 (2) edge (7)
 (2) edge (8)
 (3) edge (4)
 (3) edge (5)
 (3) edge (6)
 (3) edge (7)
 (3) edge (8)
 (4) edge (5)
 (4) edge (6)
 (4) edge (7)
 (4) edge (8)
 (5) edge (6)
 (5) edge (7)
 (5) edge (8)
 (6) edge (7)
 (6) edge (8)
 (7) edge (8);
\end{tikzpicture}
}
&
\center{
\begin{tikzpicture}[font=\small,baseline=-20]
\node[cl,label={right:$8$}] (1) at (0,-1) {};
\path[-]
 (1) edge [loopup] (1);
\end{tikzpicture}
}\endline
\hline
\center{
\begin{tikzpicture}[scale = .4,baseline=-10]
\node[vertex] (1) at (0,1) {};
\node[vertex] (2) at (2,0) {};
\node[vertex] (3) at (4,0) {};
\node[vertex] (4) at (6,1) {};
\node[vertex] (5) at (0,2) {};
\node[vertex] (6) at (2,3) {};
\node[vertex] (7) at (4,3) {};
\node[vertex] (8) at (6,2) {};
\path[-]
 (1) edge (2)
 (1) edge (3)
 (1) edge (4)
 (1) edge (5)
 (1) edge (6)
 (1) edge (7)
 (1) edge (8)
 (2) edge (4)
 (2) edge (5)
 (2) edge (6)
 (2) edge (7)
 (2) edge (8)
 (3) edge (4)
 (3) edge (5)
 (3) edge (6)
 (3) edge (7)
 (3) edge (8)
 (4) edge (5)
 (4) edge (6)
 (4) edge (7)
 (4) edge (8)
 (5) edge (6)
 (5) edge (7)
 (5) edge (8)
 (6) edge (7)
 (6) edge (8)
 (7) edge (8);
\end{tikzpicture}
}
&
\center{
\begin{tikzpicture}[font=\small,baseline=5]
\node[cl,label={below:$6$}] (1) at (0,0) {};
\node[op,label={below:$2$}] (2) at (1,0) {};
\path[-]
 (1) edge [loopup] (1)
 (1) edge (2);
\end{tikzpicture}
}\endline
\hline
\center{
\begin{tikzpicture}[scale = .4,baseline=-10]
\node[vertex] (1) at (0,1) {};
\node[vertex] (2) at (2,0) {};
\node[vertex] (3) at (4,0) {};
\node[vertex] (4) at (6,1) {};
\node[vertex] (5) at (0,2) {};
\node[vertex] (6) at (2,3) {};
\node[vertex] (7) at (4,3) {};
\node[vertex] (8) at (6,2) {};
\path[-]
 (1) edge (2)
 (1) edge (3)
 (2) edge (4)
 (2) edge (5)
 (2) edge (6)
 (2) edge (7)
 (2) edge (8)
 (3) edge (4)
 (3) edge (5)
 (3) edge (6)
 (3) edge (7)
 (3) edge (8);
\end{tikzpicture}
}
&
\center{
\begin{tikzpicture}[font=\small,baseline=5]
\node[op,label={below:$6$}] (1) at (0,0) {};
\node[op,label={below:$2$}] (2) at (1,0) {};
\path[-]
 (1) edge (2);
\end{tikzpicture}
}\endline
\hline
\center{
\begin{tikzpicture}[scale = .4,baseline=-10]
\node[vertex] (1) at (0,1) {};
\node[vertex] (2) at (2,0) {};
\node[vertex] (3) at (4,0) {};
\node[vertex] (4) at (6,1) {};
\node[vertex] (5) at (0,2) {};
\node[vertex] (6) at (2,3) {};
\node[vertex] (7) at (4,3) {};
\node[vertex] (8) at (6,2) {};
\path[-]
 (1) edge (4)
 (1) edge (5)
 (1) edge (6)
 (1) edge (7)
 (1) edge (8)
 (2) edge (4)
 (2) edge (5)
 (2) edge (6)
 (2) edge (7)
 (2) edge (8)
 (3) edge (4)
 (3) edge (5)
 (3) edge (6)
 (3) edge (7)
 (3) edge (8)
 (4) edge (5)
 (4) edge (6)
 (4) edge (7)
 (4) edge (8)
 (5) edge (6)
 (5) edge (7)
 (5) edge (8)
 (6) edge (7)
 (6) edge (8)
 (7) edge (8);
\end{tikzpicture}
}
&
\center{
\begin{tikzpicture}[font=\small,baseline=5]
\node[cl,label={below:$5$}] (1) at (0,0) {};
\node[op,label={below:$3$}] (2) at (1,0) {};
\path[-]
 (1) edge [loopup] (1)
 (1) edge (2);
\end{tikzpicture}
}\endline
\hline
\center{
\begin{tikzpicture}[scale = .4,baseline=-10]
\node[vertex] (1) at (0,1) {};
\node[vertex] (2) at (2,0) {};
\node[vertex] (3) at (4,0) {};
\node[vertex] (4) at (6,1) {};
\node[vertex] (5) at (0,2) {};
\node[vertex] (6) at (2,3) {};
\node[vertex] (7) at (4,3) {};
\node[vertex] (8) at (6,2) {};
\path[-]
 (1) edge (6)
 (1) edge (7)
 (1) edge (8)
 (2) edge (6)
 (2) edge (7)
 (2) edge (8)
 (3) edge (6)
 (3) edge (7)
 (3) edge (8)
 (4) edge (6)
 (4) edge (7)
 (4) edge (8)
 (5) edge (6)
 (5) edge (7)
 (5) edge (8)
 (6) edge (7)
 (6) edge (8)
 (7) edge (8);
\end{tikzpicture}
}
&
\center{
\begin{tikzpicture}[font=\small,baseline=5]
\node[op,label={below:$5$}] (1) at (0,0) {};
\node[cl,label={below:$3$}] (2) at (1,0) {};
\path[-]
 (2) edge [loopup] (2)
 (1) edge (2);
\end{tikzpicture}
}\endline
\hline
\center{
\begin{tikzpicture}[scale = .4,baseline=-10]
\node[vertex] (1) at (0,1) {};
\node[vertex] (2) at (2,0) {};
\node[vertex] (3) at (4,0) {};
\node[vertex] (4) at (6,1) {};
\node[vertex] (5) at (0,2) {};
\node[vertex] (6) at (2,3) {};
\node[vertex] (7) at (4,3) {};
\node[vertex] (8) at (6,2) {};
\path[-]
 (1) edge (4)
 (1) edge (5)
 (1) edge (6)
 (1) edge (7)
 (1) edge (8)
 (2) edge (4)
 (2) edge (5)
 (2) edge (6)
 (2) edge (7)
 (2) edge (8)
 (3) edge (4)
 (3) edge (5)
 (3) edge (6)
 (3) edge (7)
 (3) edge (8);
\end{tikzpicture}
}
&
\center{
\begin{tikzpicture}[font=\small,baseline=5]
\node[op,label={below:$5$}] (1) at (0,0) {};
\node[op,label={below:$3$}] (2) at (1,0) {};
\path[-]
 (1) edge (2);
\end{tikzpicture}
}\endline
\hline
\center{
\begin{tikzpicture}[scale = .4,baseline=-10]
\node[vertex] (1) at (0,1) {};
\node[vertex] (2) at (2,0) {};
\node[vertex] (3) at (4,0) {};
\node[vertex] (4) at (6,1) {};
\node[vertex] (5) at (0,2) {};
\node[vertex] (6) at (2,3) {};
\node[vertex] (7) at (4,3) {};
\node[vertex] (8) at (6,2) {};
\path[-]
 (1) edge (5)
 (1) edge (6)
 (1) edge (7)
 (1) edge (8)
 (2) edge (5)
 (2) edge (6)
 (2) edge (7)
 (2) edge (8)
 (3) edge (5)
 (3) edge (6)
 (3) edge (7)
 (3) edge (8)
 (4) edge (5)
 (4) edge (6)
 (4) edge (7)
 (4) edge (8)
 (5) edge (6)
 (5) edge (7)
 (5) edge (8)
 (6) edge (7)
 (6) edge (8)
 (7) edge (8);
\end{tikzpicture}
}
&
\center{
\begin{tikzpicture}[font=\small,baseline=5]
\node[cl,label={below:$4$}] (1) at (0,0) {};
\node[op,label={below:$4$}] (2) at (1,0) {};
\path[-]
 (1) edge [loopup] (1)
 (1) edge (2);
\end{tikzpicture}
}\endline
\hline
\center{
\begin{tikzpicture}[scale = .4,baseline=-10]
\node[vertex] (1) at (0,1) {};
\node[vertex] (2) at (2,0) {};
\node[vertex] (3) at (4,0) {};
\node[vertex] (4) at (6,1) {};
\node[vertex] (5) at (0,2) {};
\node[vertex] (6) at (2,3) {};
\node[vertex] (7) at (4,3) {};
\node[vertex] (8) at (6,2) {};
\path[-]
 (1) edge (5)
 (1) edge (6)
 (1) edge (7)
 (1) edge (8)
 (2) edge (5)
 (2) edge (6)
 (2) edge (7)
 (2) edge (8)
 (3) edge (5)
 (3) edge (6)
 (3) edge (7)
 (3) edge (8)
 (4) edge (5)
 (4) edge (6)
 (4) edge (7)
 (4) edge (8);
\end{tikzpicture}
}
&
\center{
\begin{tikzpicture}[font=\small,baseline=5]
\node[op,label={below:$4$}] (1) at (0,0) {};
\node[op,label={below:$4$}] (2) at (1,0) {};
\path[-]
 (1) edge (2);
\end{tikzpicture}
}\endline
\hline
\center{
\begin{tikzpicture}[scale = .4,baseline=-10]
\node[vertex] (1) at (0,1) {};
\node[vertex] (2) at (2,0) {};
\node[vertex] (3) at (4,0) {};
\node[vertex] (4) at (6,1) {};
\node[vertex] (5) at (0,2) {};
\node[vertex] (6) at (2,3) {};
\node[vertex] (7) at (4,3) {};
\node[vertex] (8) at (6,2) {};
\path[-]
 (1) edge (3)
 (1) edge (4)
 (1) edge (5)
 (1) edge (6)
 (1) edge (7)
 (1) edge (8)
 (2) edge (3)
 (2) edge (4)
 (2) edge (5)
 (2) edge (6)
 (2) edge (7)
 (2) edge (8)
 (3) edge (5)
 (3) edge (6)
 (3) edge (7)
 (3) edge (8)
 (4) edge (5)
 (4) edge (6)
 (4) edge (7)
 (4) edge (8)
 (5) edge (6)
 (5) edge (7)
 (5) edge (8)
 (6) edge (7)
 (6) edge (8)
 (7) edge (8);
\end{tikzpicture}
}
&
\center{
\begin{tikzpicture}[font=\small,baseline=-10]
\node[op,label={below:$2$}] (1) at (0,-1) {};
\node[cl,label={below:$4$}] (2) at (1,0) {};
\node[op,label={below:$2$}] (3) at (2,-1) {};
\path[-]
 (2) edge [loopup] (2)
 (1) edge (2)
 (2) edge (3)
 (1) edge (3);
\end{tikzpicture}
}\endline
\hline
\center{
\begin{tikzpicture}[scale = .4,baseline=-10]
\node[vertex] (1) at (0,1) {};
\node[vertex] (2) at (2,0) {};
\node[vertex] (3) at (4,0) {};
\node[vertex] (4) at (6,1) {};
\node[vertex] (5) at (0,2) {};
\node[vertex] (6) at (2,3) {};
\node[vertex] (7) at (4,3) {};
\node[vertex] (8) at (6,2) {};
\path[-]
 (1) edge (3)
 (1) edge (4)
 (1) edge (5)
 (1) edge (6)
 (1) edge (7)
 (1) edge (8)
 (2) edge (3)
 (2) edge (4)
 (2) edge (5)
 (2) edge (6)
 (2) edge (7)
 (2) edge (8)
 (3) edge (5)
 (3) edge (6)
 (3) edge (7)
 (3) edge (8)
 (4) edge (5)
 (4) edge (6)
 (4) edge (7)
 (4) edge (8);
\end{tikzpicture}
}
&
\center{
\begin{tikzpicture}[font=\small,baseline=-10]
\node[op,label={below:$2$}] (1) at (0,-1) {};
\node[op,label={below:$4$}] (2) at (1,0) {};
\node[op,label={below:$2$}] (3) at (2,-1) {};
\path[-]
 (1) edge (2)
 (2) edge (3)
 (1) edge (3);
\end{tikzpicture}
}\endline
\hline
\center{
\begin{tikzpicture}[scale = .4,baseline=-10]
\node[vertex] (1) at (0,1) {};
\node[vertex] (2) at (2,0) {};
\node[vertex] (3) at (4,0) {};
\node[vertex] (4) at (6,1) {};
\node[vertex] (5) at (0,2) {};
\node[vertex] (6) at (2,3) {};
\node[vertex] (7) at (4,3) {};
\node[vertex] (8) at (6,2) {};
\path[-]
 (1) edge (3)
 (1) edge (4)
 (1) edge (5)
 (1) edge (6)
 (1) edge (7)
 (1) edge (8)
 (2) edge (3)
 (2) edge (4)
 (2) edge (5)
 (2) edge (6)
 (2) edge (7)
 (2) edge (8)
 (5) edge (6)
 (5) edge (7)
 (5) edge (8)
 (6) edge (7)
 (6) edge (8)
 (7) edge (8);
\end{tikzpicture}
}
&
\center{
\begin{tikzpicture}[font=\small,baseline=-25]
\node[cl,label={below:$4$}] (1) at (0,-1) {};
\node[op,label={below:$2$}] (2) at (1,-1) {};
\node[op,label={below:$2$}] (3) at (2,-1) {};
\path[-]
 (1) edge [loopup] (1)
 (1) edge (2)
 (2) edge (3);
\end{tikzpicture}
}\endline
\hline
\center{
\begin{tikzpicture}[scale = .4,baseline=-10]
\node[vertex] (1) at (0,1) {};
\node[vertex] (2) at (2,0) {};
\node[vertex] (3) at (4,0) {};
\node[vertex] (4) at (6,1) {};
\node[vertex] (5) at (0,2) {};
\node[vertex] (6) at (2,3) {};
\node[vertex] (7) at (4,3) {};
\node[vertex] (8) at (6,2) {};
\path[-]
 (1) edge (2)
 (1) edge (3)
 (1) edge (4)
 (1) edge (5)
 (2) edge (3)
 (2) edge (4)
 (2) edge (5)
 (3) edge (4)
 (3) edge (5)
 (4) edge (6)
 (4) edge (7)
 (4) edge (8)
 (5) edge (6)
 (5) edge (7)
 (5) edge (8)
 (6) edge (7)
 (6) edge (8)
 (7) edge (8);
\end{tikzpicture}
}
&
\center{
\begin{tikzpicture}[font=\small,baseline=-25]
\node[cl,label={below:$3$}] (1) at (0,-1) {};
\node[op,label={below:$2$}] (2) at (1,-1) {};
\node[cl,label={below:$3$}] (3) at (2,-1) {};
\path[-]
 (1) edge [loopup] (1)
 (3) edge [loopup] (3)
 (1) edge (2)
 (2) edge (3);
\end{tikzpicture}
}\endline
\hline
\center{
\begin{tikzpicture}[scale = .4,baseline=-10]
\node[vertex] (1) at (0,1) {};
\node[vertex] (2) at (2,0) {};
\node[vertex] (3) at (4,0) {};
\node[vertex] (4) at (6,1) {};
\node[vertex] (5) at (0,2) {};
\node[vertex] (6) at (2,3) {};
\node[vertex] (7) at (4,3) {};
\node[vertex] (8) at (6,2) {};
\path[-]
 (1) edge (2)
 (1) edge (3)
 (1) edge (4)
 (1) edge (5)
 (2) edge (3)
 (2) edge (4)
 (2) edge (5)
 (3) edge (4)
 (3) edge (5)
 (4) edge (6)
 (4) edge (7)
 (4) edge (8)
 (5) edge (6)
 (5) edge (7)
 (5) edge (8);
\end{tikzpicture}
}
&
\center{
\begin{tikzpicture}[font=\small,baseline=-25]
\node[cl,label={below:$3$}] (1) at (0,-1) {};
\node[op,label={below:$2$}] (2) at (1,-1) {};
\node[op,label={below:$3$}] (3) at (2,-1) {};
\path[-]
 (1) edge [loopup] (1)
 (1) edge (2)
 (2) edge (3);
\end{tikzpicture}
}\endline
\hline
\center{
\begin{tikzpicture}[scale = .4,baseline=-10]
\node[vertex] (1) at (0,1) {};
\node[vertex] (2) at (2,0) {};
\node[vertex] (3) at (4,0) {};
\node[vertex] (4) at (6,1) {};
\node[vertex] (5) at (0,2) {};
\node[vertex] (6) at (2,3) {};
\node[vertex] (7) at (4,3) {};
\node[vertex] (8) at (6,2) {};
\path[-]
 (1) edge (2)
 (1) edge (3)
 (1) edge (4)
 (1) edge (5)
 (1) edge (6)
 (2) edge (3)
 (2) edge (4)
 (2) edge (5)
 (2) edge (6)
 (3) edge (4)
 (3) edge (5)
 (3) edge (6)
 (4) edge (5)
 (4) edge (6)
 (4) edge (7)
 (4) edge (8)
 (5) edge (6)
 (5) edge (7)
 (5) edge (8)
 (6) edge (7)
 (6) edge (8);
\end{tikzpicture}
}
&
\center{
\begin{tikzpicture}[font=\small,baseline=-25]
\node[cl,label={below:$3$}] (1) at (0,-1) {};
\node[cl,label={below:$3$}] (2) at (1,-1) {};
\node[op,label={below:$2$}] (3) at (2,-1) {};
\path[-]
 (1) edge [loopup] (1)
 (2) edge [loopup] (2)
 (1) edge (2)
 (2) edge (3);
\end{tikzpicture}
}\endline
\hline
\center{
\begin{tikzpicture}[scale = .4,baseline=-10]
\node[vertex] (1) at (0,1) {};
\node[vertex] (2) at (2,0) {};
\node[vertex] (3) at (4,0) {};
\node[vertex] (4) at (6,1) {};
\node[vertex] (5) at (0,2) {};
\node[vertex] (6) at (2,3) {};
\node[vertex] (7) at (4,3) {};
\node[vertex] (8) at (6,2) {};
\path[-]
 (1) edge (2)
 (1) edge (3)
 (1) edge (4)
 (1) edge (5)
 (1) edge (6)
 (2) edge (3)
 (2) edge (4)
 (2) edge (5)
 (2) edge (6)
 (3) edge (4)
 (3) edge (5)
 (3) edge (6)
 (4) edge (7)
 (4) edge (8)
 (5) edge (7)
 (5) edge (8)
 (6) edge (7)
 (6) edge (8);
\end{tikzpicture}
}
&
\center{
\begin{tikzpicture}[font=\small,baseline=-25]
\node[cl,label={below:$3$}] (1) at (0,-1) {};
\node[op,label={below:$3$}] (2) at (1,-1) {};
\node[op,label={below:$2$}] (3) at (2,-1) {};
\path[-]
 (1) edge [loopup] (1)
 (1) edge (2)
 (2) edge (3);
\end{tikzpicture}
}\endline
\hline
\center{
\begin{tikzpicture}[scale = .4,baseline=-10]
\node[vertex] (1) at (0,1) {};
\node[vertex] (2) at (2,0) {};
\node[vertex] (3) at (4,0) {};
\node[vertex] (4) at (6,1) {};
\node[vertex] (5) at (0,2) {};
\node[vertex] (6) at (2,3) {};
\node[vertex] (7) at (4,3) {};
\node[vertex] (8) at (6,2) {};
\path[-]
 (1) edge (2)
 (1) edge (3)
 (1) edge (4)
 (1) edge (5)
 (1) edge (6)
 (1) edge (7)
 (1) edge (8)
 (2) edge (3)
 (2) edge (4)
 (2) edge (5)
 (2) edge (6)
 (2) edge (7)
 (2) edge (8)
 (3) edge (4)
 (3) edge (5)
 (3) edge (6)
 (3) edge (7)
 (3) edge (8)
 (4) edge (5)
 (4) edge (6)
 (4) edge (7)
 (5) edge (8)
 (6) edge (8)
 (7) edge (8);
\end{tikzpicture}
}
&
\center{
\begin{tikzpicture}[font=\small,baseline=-10]
\node[cl,label={below:$3$}] (1) at (0,-1) {};
\node[op,label={below:$3$}] (2) at (1,0) {};
\node[op,label={below:$2$}] (3) at (2,-1) {};
\path[-]
 (1) edge [loopup] (1)
 (1) edge (2)
 (2) edge (3)
 (1) edge (3);
\end{tikzpicture}
}\endline
\hline
\center{
\begin{tikzpicture}[scale = .4,baseline=-10]
\node[vertex] (1) at (0,1) {};
\node[vertex] (2) at (2,0) {};
\node[vertex] (3) at (4,0) {};
\node[vertex] (4) at (6,1) {};
\node[vertex] (5) at (0,2) {};
\node[vertex] (6) at (2,3) {};
\node[vertex] (7) at (4,3) {};
\node[vertex] (8) at (6,2) {};
\path[-]
 (1) edge (4)
 (1) edge (5)
 (1) edge (6)
 (1) edge (7)
 (1) edge (8)
 (2) edge (4)
 (2) edge (5)
 (2) edge (6)
 (2) edge (7)
 (2) edge (8)
 (3) edge (4)
 (3) edge (5)
 (3) edge (6)
 (3) edge (7)
 (3) edge (8)
 (4) edge (5)
 (4) edge (6)
 (4) edge (7)
 (5) edge (8)
 (6) edge (8)
 (7) edge (8);
\end{tikzpicture}
}
&
\center{
\begin{tikzpicture}[font=\small,baseline=-10]
\node[op,label={below:$3$}] (1) at (0,-1) {};
\node[op,label={below:$3$}] (2) at (1,0) {};
\node[op,label={below:$2$}] (3) at (2,-1) {};
\path[-]
 (1) edge (2)
 (2) edge (3)
 (1) edge (3);
\end{tikzpicture}
}\endline
\hline
\center{
\begin{tikzpicture}[scale = .4,baseline=-10]
\node[vertex] (1) at (0,1) {};
\node[vertex] (2) at (2,0) {};
\node[vertex] (3) at (4,0) {};
\node[vertex] (4) at (6,1) {};
\node[vertex] (5) at (0,2) {};
\node[vertex] (6) at (2,3) {};
\node[vertex] (7) at (4,3) {};
\node[vertex] (8) at (6,2) {};
\path[-]
 (1) edge (3)
 (1) edge (4)
 (1) edge (5)
 (1) edge (6)
 (2) edge (3)
 (2) edge (4)
 (2) edge (5)
 (2) edge (6)
 (3) edge (5)
 (3) edge (6)
 (4) edge (5)
 (4) edge (6)
 (5) edge (7)
 (5) edge (8)
 (6) edge (7)
 (6) edge (8);
\end{tikzpicture}
}
&
\center{
\begin{tikzpicture}[font=\small,baseline=-10]
\node[op,label={left:$2$}] (1) at (0,-1) {};
\node[op,label={right:$2$}] (2) at (1,-1) {};
\node[op,label={left:$2$}] (3) at (0,0) {};
\node[op,label={right:$2$}] (4) at (1,0) {};
\path[-]
 (1) edge (2)
 (2) edge (3)
 (1) edge (3)
 (3) edge (4);
\end{tikzpicture}
}
\endline
\hline
\center{
\begin{tikzpicture}[scale = .4,baseline=-10]
\node[vertex] (1) at (0,1) {};
\node[vertex] (2) at (2,0) {};
\node[vertex] (3) at (4,0) {};
\node[vertex] (4) at (6,1) {};
\node[vertex] (5) at (0,2) {};
\node[vertex] (6) at (2,3) {};
\node[vertex] (7) at (4,3) {};
\node[vertex] (8) at (6,2) {};
\path[-]
 (1) edge (3)
 (1) edge (4)
 (2) edge (3)
 (2) edge (4)
 (3) edge (5)
 (3) edge (6)
 (4) edge (5)
 (4) edge (6)
 (5) edge (7)
 (5) edge (8)
 (6) edge (7)
 (6) edge (8);
\end{tikzpicture}
}
&
\center{
\begin{tikzpicture}[font=\small,baseline=-25]
\node[op,label={below:$2$}] (1) at (0,-1) {};
\node[op,label={below:$2$}] (2) at (1,-1) {};
\node[op,label={below:$2$}] (3) at (2,-1) {};
\node[op,label={below:$2$}] (4) at (3,-1) {};
\path[-]
 (1) edge (2)
 (2) edge (3)
 (3) edge (4);
\end{tikzpicture}
}\endline
\hline
\center{
\begin{tikzpicture}[scale = .4,baseline=-10]
\node[vertex] (1) at (0,1) {};
\node[vertex] (2) at (2,0) {};
\node[vertex] (3) at (4,0) {};
\node[vertex] (4) at (6,1) {};
\node[vertex] (5) at (0,2) {};
\node[vertex] (6) at (2,3) {};
\node[vertex] (7) at (4,3) {};
\node[vertex] (8) at (6,2) {};
\path[-]
 (1) edge (3)
 (1) edge (4)
 (1) edge (5)
 (1) edge (6)
 (1) edge (7)
 (1) edge (8)
 (2) edge (3)
 (2) edge (4)
 (2) edge (5)
 (2) edge (6)
 (2) edge (7)
 (2) edge (8)
 (3) edge (5)
 (3) edge (6)
 (3) edge (7)
 (3) edge (8)
 (4) edge (5)
 (4) edge (6)
 (4) edge (7)
 (4) edge (8)
 (5) edge (7)
 (5) edge (8)
 (6) edge (7)
 (6) edge (8);
\end{tikzpicture}
}
&
\center{
\begin{tikzpicture}[font=\small,baseline=-10]
\node[op,label={left:$2$}] (1) at (0,-1) {};
\node[op,label={right:$2$}] (2) at (1,-1) {};
\node[op,label={left:$2$}] (3) at (0,0) {};
\node[op,label={right:$2$}] (4) at (1,0) {};
\path[-]
 (1) edge (2)
 (1) edge (3)
 (1) edge (4)
 (2) edge (3)
 (2) edge (4)
 (3) edge (4);
\end{tikzpicture}
}\endline
\hline
\caption{Anosov graphs on eight vertices and their corresponding quotient graphs induced by $\sim$.}\label{anosov_list_8}
\end{longtable}

\section{Anosov test in \sage}\label{appB}

\vspace*{-10pt}
\begin{lstlisting}
def is_Anosov(self, allow_disconnected=False):
    r"""
    Return ``True`` if ``self`` is Anosov.
    """
    n = self.order()
    L = range(n)
    classes = []
    V = self.vertices()
    if self.is_connected() == False and allow_disconnected == False:
        return False
    for v in range(n):
        N_v = self.neighbors(V[v])
        CN_v = self.neighbors(V[v]) + [v]
        CN_v.sort()
        L[v] = [v]
        for u in range(n):
            N_u = self.neighbors(V[u])
            CN_u = self.neighbors(V[u]) + [u]
            CN_u.sort()
            if ((N_v == N_u) or (CN_v == CN_u)) and (v != u):
                L[v].append(u)
                L[v].sort()
        if (L[v] not in classes):
            classes.append(L[v])
    for i in range(len(classes)):
        if len(classes[i]) == 1:
            return False
        if len(classes[i]) == 2:
            for j in classes[i]:
                for k in classes[i]:
                    if self.neighbors(V[j]) != self.neighbors(V[k]):
                        return False
    return True
\end{lstlisting}

\section*{Acknowledgements}
The authors would like to thank Mohan Shrikhande for a careful reading of a preliminary version of this manuscript, which was the second named author's master's thesis.  The authors would also like to thank Dave Morris for very helpful suggestions on an earlier draft about the presentation of this material.  Finally, the authors thank the developers of \sage\ \cite{sage} for making their lives' much easier via the excellent graph support in \sage.

\end{document}